\documentclass[12pt]{amsart}

\usepackage{amssymb}

\usepackage{enumitem}

\usepackage{graphicx}

\makeatletter
\@namedef{subjclassname@2020}{%
  \textup{2020} Mathematics Subject Classification}
\makeatother

\usepackage[T1]{fontenc}

\newtheorem{theorem}{Theorem}[section]
\newtheorem{corollary}[theorem]{Corollary}

\newtheorem{proposition}[theorem]{Proposition}

\theoremstyle{definition}
\newtheorem{definition}[theorem]{Definition}
\newtheorem{remark}[theorem]{Remark}
\newtheorem{example}[theorem]{Example}

\numberwithin{equation}{section}

\begin{document}


\baselineskip=17pt


\title[Curlicues generated by circle homeomorphisms]{Curlicues generated by circle homeomorphisms}

\author[J. Signerska-Rynkowska]{Justyna Signerska-Rynkowska}
\address{Faculty of Applied Physics and Mathematics \\ Gda\'nsk University of Technology \\
Narutowicza 11/12\\
80-299 Gda{\'n}sk, Poland}
\email{justyna.signerska@pg.edu.pl}

\date{}

\begin{abstract}
We investigate the curves in the complex plane which are generated by sequences of real numbers being the lifts of the points on the orbit of an orientation preserving circle homeomorphism. Geometrical properties of these curves such as boundedness, superficiality, local discrete radius of curvature are linked with dynamical properties of the circle homeomorphism which generates them: rotation number and its continued fraction expansion, existence of a continuous solution of the corresponding cohomological equation and displacement sequence along the orbit. 
\end{abstract}



\maketitle

\section{Introduction}
\label{intro}
The term \emph{curlicue} is probably mostly used in various visual arts, for example it can be a recurring decorative motif in architecture, calligraphy or fashion design. In this article we look at mathematical curlicues: 
\begin{definition}\label{defin0}
A curlicue $\Gamma=\Gamma(u)$, where $u=(u_n)_{n=0}^{\infty}\subset \mathbb{R}$, is a piece-wise linear curve in $\mathbb{C}$ passing consecutively through the points $z_0=0 \in \mathbb{C}$, and $z_1$, $z_2$, ..., where
\begin{equation}\label{glowne1}
  z_n = \sum_{k=0}^{n-1}\exp(2\pi\imath u_k), \quad n=1,2,...
\end{equation}
In other words,
\begin{equation}\label{glowne2}
  z_n = z_{n-1}+\exp(2\pi\imath u_{n-1}), \quad n=1,2,...
\end{equation}
\end{definition}

A curlicue can be obtained from an arbitrary sequence $(u_n)_{n=0}^{\infty}$ of real numbers. However, in this paper we assume that $u_n:=\Phi^n(x_0)$, $n=0,1,...$, $x_0\in\mathbb{R}$, with $\Phi:\mathbb{R}\to\mathbb{R}$ being a lift of an orientation preserving circle homeomorphism $\varphi:S^1\to S^1$, where $\mathbb{R}$ covers $S^1$ via the standard projection: $\mathfrak{p}:\mathbb{R}\to S^1$, $\mathfrak{p}(x)=\exp(2\pi\imath x)$. Construction of such a curlicue is illustrated in Figure \ref{rys1}. Sometimes we will also denote $\Gamma$ as $\Gamma((u_n))$ and when the generating homeomorphism $\Phi$ is clear from the context, we will write $\Gamma(x_0)$ to distinguish between the curves generated by the same homeomorphism but along the orbits of different initial points $x_0$.  
\begin{figure}
  \centering
    \includegraphics[width=0.5\textwidth]{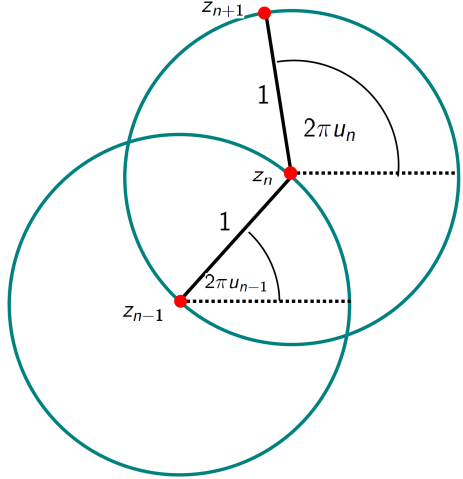}
     \caption{Construction of a curlicue}\label{rys1}
\end{figure}

The name \emph{curlicue} for such a curve is not accidentally connected  with the artistic notion of a curlicue: indeed, these curves, obtained for various sequences $(u_n)_{n=0}^{\infty}\subset\mathbb{R}$, can form beautiful shapes as one can see, for example, in the papers of Dekking and Mend\`{e}s-France (\cite{curl2}), who studied geometrical properties of such curves (superficiality and dimension), Berry and Golberg (\cite{berry}), Sinai (\cite{curl}) or Cellarosi (\cite{cellarosi}) who studied and developed techniques of renormalisation and limiting distributions of classical curlicues, i.e. for $u_n=\alpha n^2$. Many fantastic pictures of curlicues can be found also in the work of Moore and van der Poorten (\cite{moore}), who gave nice description of the work \cite{berry}. However, we would like to draw attention to dynamically generated curlicues, i.e. the curves $\Gamma$, where $(u_n)$ is obtained from an orbit of a given map $f$ since reflecting the dynamics of $f$ in the structure of $\Gamma$ might be in general an intriguing question. 

We also remark that in the existing literature the term curlicues (if used at all) often refers to spiral-like components of the curve $\Gamma$ (which usually has both straight-like and spiral-like parts). However, in the current paper by a  ``curlicue'' we mean the whole curve $\Gamma$, defined as above. Perhaps it is also worth mentioning that the  curlicue  can be interpreted as a trajectory of a particle in the plane which starts in the origin at time $t=0$ and moves with a constant velocity, changing its direction at instances $t=0,1,2,3,...$, where the new direction is given by a number $2\pi u_t\in [0, 2\pi)$ (as mentioned e.g. in \cite{curl2}). Thus $\Gamma$ can be seen as a trajectory of a walk obtained through some dynamical system (compare, for example, with \cite{avila2}).

In this study we are mainly interested not in ``ergodic'' but rather in geometric properties of curlicues such as boundedness and superficiality (defined below). Although dynamics of circle homeomorphisms is now well understood  (see e.g. \cite{katok}), it turns out that it is not so trivial to give complete description of curlicues determined by them. In section \ref{secRati} we prove that geometric properties of such curves are inevitably connected with rationality of the rotation number $\varrho$ of the circle homeomorphism $\varphi$.  However, unless $\varphi$ is a rigid rotation, this relation is not so straightforward. In particular, there are no simple criteria for deciding whether $\Gamma$ is bounded or not (even for equidistributed sequences $(u_n)_{n=0}^{\infty}$, see \cite{curl2}). In section \ref{secConn} it is deduced that for $\varrho\in \mathbb{R}\setminus\mathbb{Q}$ boundedness  and shape of $\Gamma$ depend on the solution of the corresponding cohomological equation. Further, in section \ref{secGrow} we estimate  growth rate and superficiality of an unbounded curve $\Gamma$ with $\varrho\in \mathbb{R}\setminus\mathbb{Q}$ satisfying some further (generic) properties. The last sections are devoted to a local discrete radius of curvature and a brief discussion of our results and possible extensions.    

\section{Rational vs. irrational rotation number}\label{secRati}

In \cite[Example 4.1]{curl2} the following result was stated for $\varphi$ being the rotation by $\varrho$:
\begin{proposition}\label{rotation}
Let $u_n:=n\varrho$. Then
\begin{equation}\label{modul}
\vert z_n\vert=\left \vert \frac{\sin n\pi\varrho}{\sin \pi \varrho}\right\vert
\end{equation}
and the points $z_n$ lie on a circle with radius
\begin{equation}\label{radius}
R=\frac{1}{2\vert\sin{\pi \varrho\vert}}
\end{equation}
and center
\begin{equation}\label{center}
C=(\frac{1}{2},\frac{1}{2}\cot {\pi\varrho}).
\end{equation}
Furthermore,

\begin{enumerate}[label=\upshape(\roman*), leftmargin=*, widest=iii]
  \item if $\varrho\in \mathbb{Q}$ (and $\varrho\neq 0 \mod 1$), then $\Gamma(u)$ is a regular polygon (convex or star) with $q$ sides, where $\varrho=p/q$ ($p$ and $q$ relatively prime);
  \item if $\varrho\in\mathbb{R}\setminus\mathbb{Q}$, then $\Gamma(u)$ is dense in an annulus with radii
\begin{equation}\label{radii}
r_1=\frac{1}{2}\vert\cot \pi\varrho\vert  \qquad \textrm{and} \qquad r_2=\frac{1}{2\vert\sin{\pi \varrho\vert}}
\end{equation}
and
\begin{equation}\label{dim}
\mathrm{dim}\Gamma=2.
\end{equation}
\end{enumerate}
\end{proposition}
For the precise definition of the dimension $\mathrm{dim}\Gamma$ see e.g. \cite{curl2}. By the regular star polygon we mean self-intersecting, equilateral equiangular polygon, which can be constructed by connecting every $p$-th point out of $q$ points regularly spaced on the circle. For example, regular star polygon in Figure \ref{fig:irrat_rot} (left) is obtained by joining every third vertex of a regular decagon until the starting vertex is reached. Regular polygons can be described by their Schl{\"a}fli symbols $\{q,p\}$ where $p\geq 2$ and $q$ are relatively prime integers:
\begin{remark} If $\Gamma$ is a curve generated by rotation $\mathcal{R}_{\varrho}$ with $\varrho=\frac{p}{q}$, then it is a regular polygon with Schl{\"a}fli symbol $\{q,p\}$.
\end{remark}
It is easy to notice that rotation numbers of the form $1/q$ and $(q-1)/q$ correspond to  $q$-sided regular convex polygons.  

Proposition \ref{rotation} deals  with the simplest situation when the curve is generated by a circle rotation $\mathcal{R}_{\varrho}$. Clearly, the properties of $\Gamma$ are determined by rationality of $\varrho$. This simple observation is a starting point for our investigations: we ask what changes if one considers slightly more general case, i.e. when $\Gamma$ is generated by an orientation preserving circle homeomorphism (we remark that all homeomorphims of $S^1$ considered here are assumed  to be orientation preserving, even if not stated directly).  

\begin{figure}[h!]
 
                \includegraphics[width=0.325\textwidth]{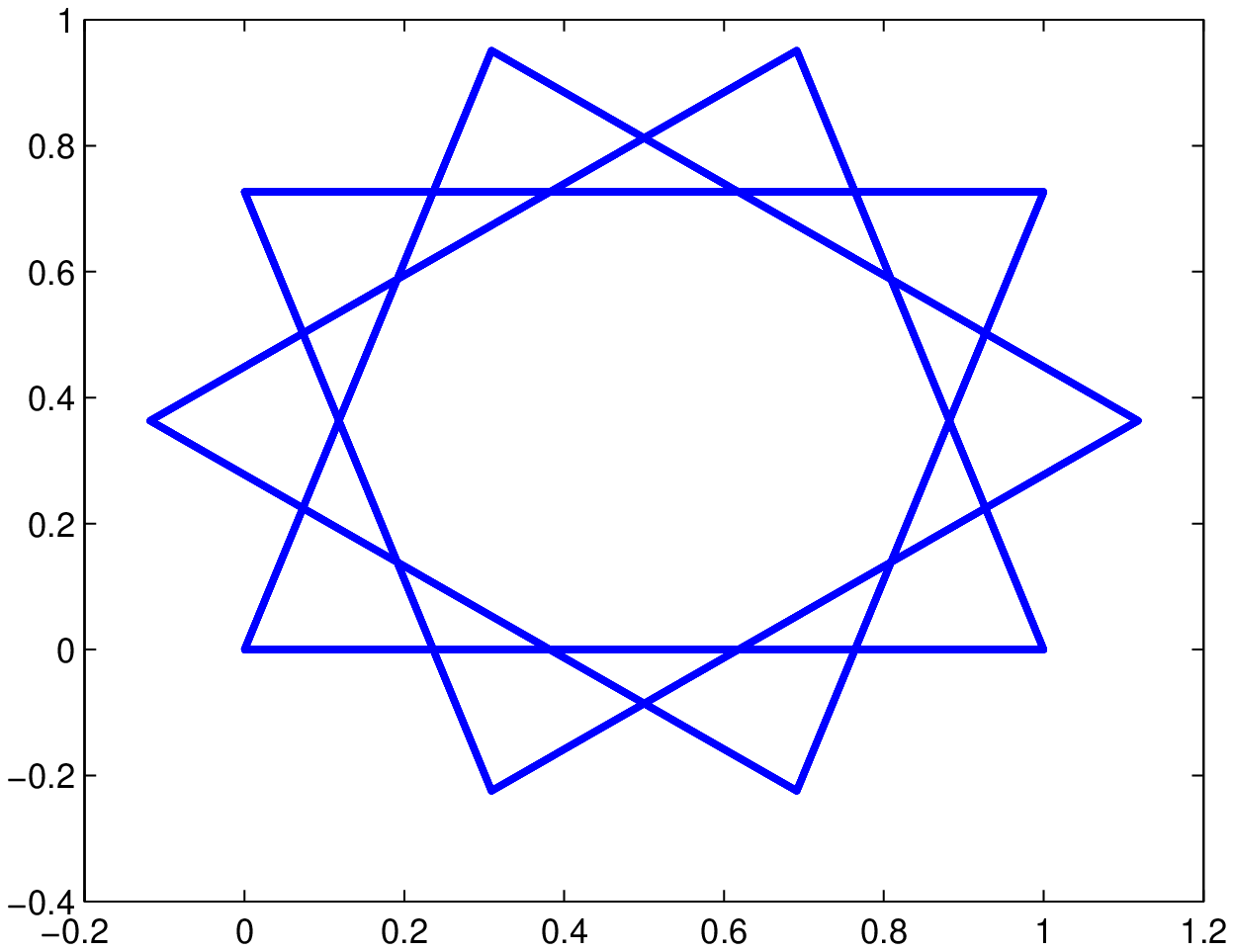}
               \includegraphics[width=0.325\textwidth]{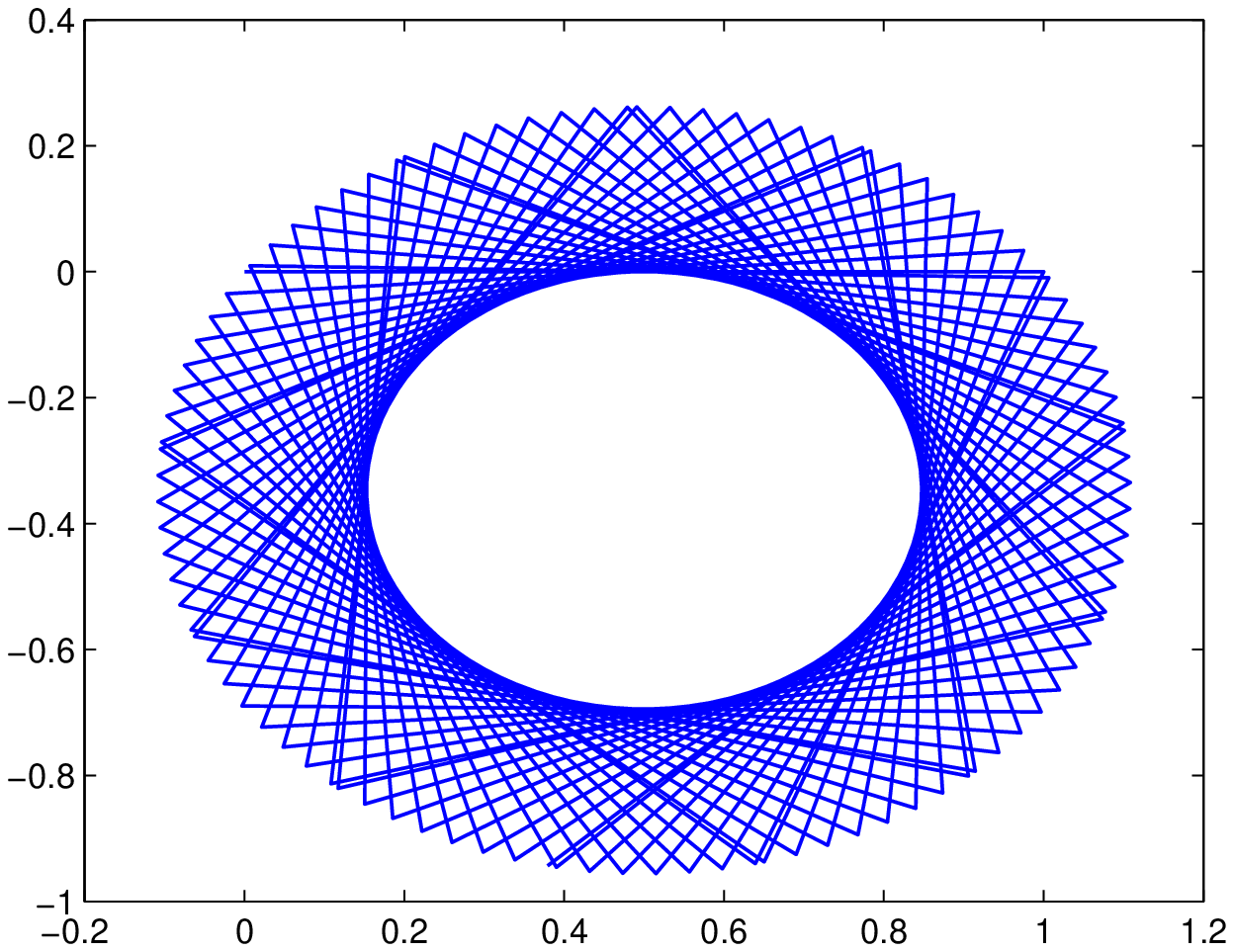}
                \includegraphics[width=0.325\textwidth]{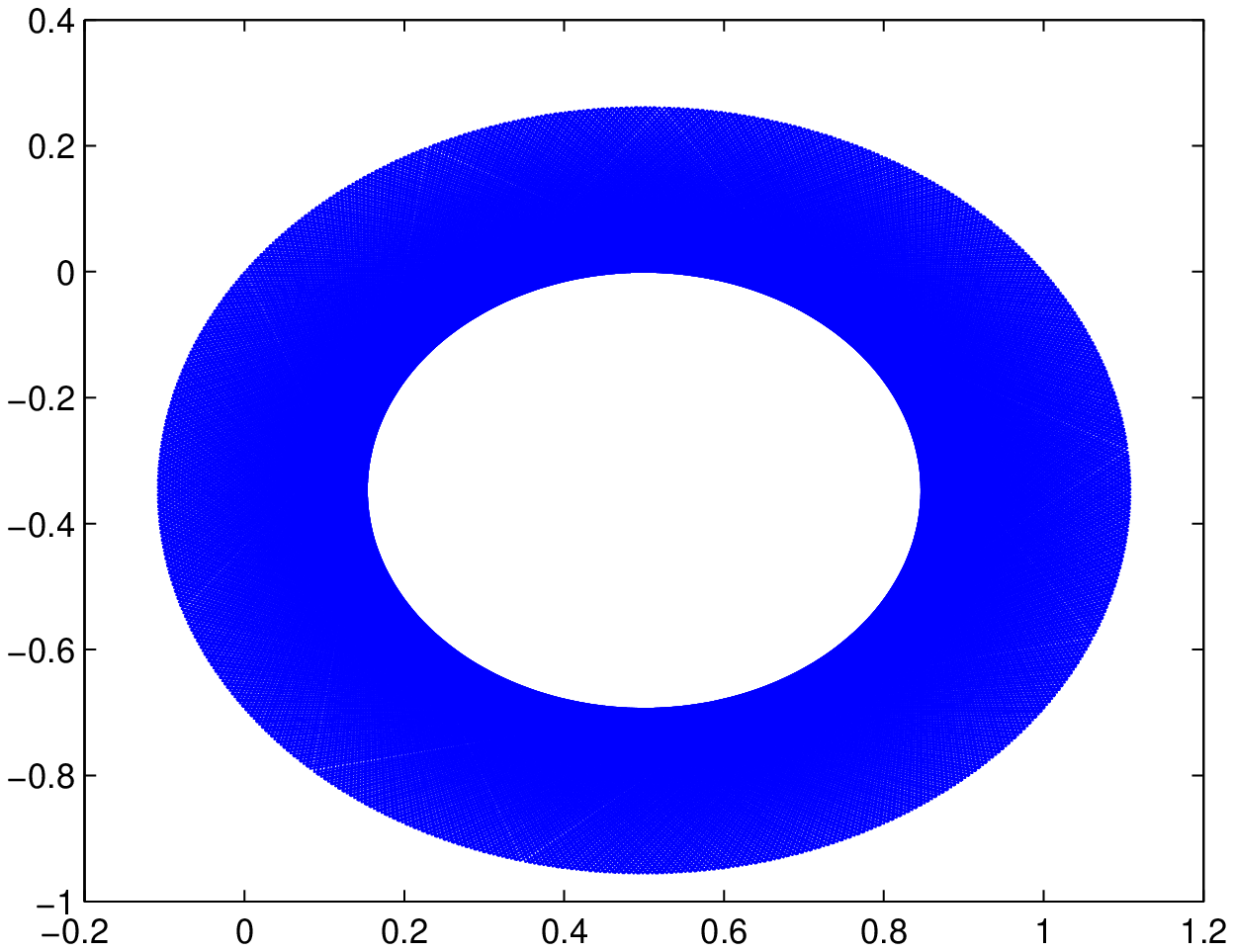}
              
       \caption{$\Gamma_{k}((n\varrho))$, a curlicue generated by the rotation $\mathcal{R}_{\varrho}$, for different values of $k$ (number of iterates) and $\varrho$ (left: $k=1000$, $\varrho=3/10$, centre: $k=100$, $\varrho=\ln 2$, right: $k=1000$, $\varrho= \ln 2$).}\label{fig:irrat_rot}
\end{figure}

Before we proceed, a few essential definitions and existing results must be recalled. 
\begin{definition}
A bounded sequence $\{u_0,u_1,u_2, ...\}$ of real numbers is \emph{equidistributed} in the interval $[a,b]$ if for any subinterval $[c,d]\subset [a,b]$ we have
\begin{displaymath}
\lim_{n\to\infty}\frac{\vert \{u_0,u_1,u_2, ...u_{n-1}\}\cap [c,d]\vert}{n}=\frac{d-c}{b -a},
\end{displaymath}
where $\vert \{u_0,u_1,u_2, ...u_{n-1}\}\cap [c,d]\vert$ denotes the number of elements of the
sequence, out of the first $n$-elements, in the interval $[c,d]$.
\end{definition}

\begin{definition}\label{equi}
The sequence $\{u_0,u_1,u_2, ...\}$ is said to be \emph{equidistributed modulo 1} (alternatively, \emph{uniformly distributed modulo 1}) if the sequence of fractional parts of its elements, i.e. the sequence $\{u_0-\left \lfloor{u_0}\right \rfloor,u_1-\left \lfloor{u_1}\right \rfloor, u_2-\left \lfloor{u_2}\right \rfloor,...\}$, is equidistributed in the interval $[0,1]$.
\end{definition}

Let us also remind that an arbitrary curve $\Gamma$ is \emph{rectifiable} if its length is finite and is said to be \emph{locally rectifiable} if all its closed subcurves are rectifiable (see e.g. \cite{ksiazkadoanalizy}). For a locally rectifiable curve  $\Gamma:=\gamma([0,\infty))$ ($\gamma:[0,\infty)\to \mathbb{R}^2$ a continuous function), we denote by $\Gamma_t$ the beginning part of $\Gamma$ which has length $t$. $\Gamma$ is called bounded if $\mathrm{Diam} \Gamma<\infty$ (otherwise, $\Gamma$ is called unbounded). For $\varepsilon>0$ we define the tabular neighborhood

\begin{displaymath}
\Gamma^{\varepsilon}:=\{y: \ \exists_{x\in\Gamma} \ d(x,y)<\varepsilon\}
\end{displaymath}

\begin{definition}\label{defsuper}
An unbounded curve  $\Gamma$ is \emph{superficial} if
\begin{displaymath}
\lim_{t\to\infty}\frac{t}{\mathrm{Diam}\Gamma_t}=\infty.
\end{displaymath}
In turn, a bounded curve $\Gamma$ is \emph{superficial} if
\begin{displaymath}
\lim_{\varepsilon\to 0}\frac{\mathrm{Area} \Gamma^{\varepsilon}}{\varepsilon}=\infty,
\end{displaymath}
where by $\mathrm{Area}$ we mean a 2-dimensional Lebesgue measure.
\end{definition}

The authors of \cite{curl2} prove a very useful criterion for a sequence to be equidistributed modulo 1.
\begin{theorem}[\cite{curl2}]\label{glowne}
Let $\Gamma=\Gamma(u)$ be a curve generated by the sequence $u=(u_n)_{n=0}^{\infty}$.

The sequence $\{u_n\}$ is equidistributed modulo $1$ if and only if for each positive integer $q$ the curve $\Gamma(qu)$ is superficial.
\end{theorem}
By $\Gamma(qu)$ we denote a curve generated by the sequence $(q u_n)_{n=0}^{\infty}$, i.e. a curve passing through the points $z_0=0\in\mathbb{C}$ and
\begin{displaymath}
z_n(q):=\sum_{k=0}^{n-1}\exp(2\pi\imath q u_k), \quad n\in\mathbb{N}.
\end{displaymath}

From the proof of Theorem 3.1 in \cite{curl2} one concludes
\begin{proposition}\label{boundedsuperficial}
If the sequence $(u_n)$ determines a bounded curve $\Gamma(u)$ and if infinitely many $u_n$ are different modulo 1, then the curve $\Gamma(u)$ is superficial.
\end{proposition}
\begin{proposition}\label{superficialforirrational}
Let $\Gamma$ be a curve generated by a circle homeomorphism $\varphi$ with an irrational rotation number $\varrho\in \mathbb{R}\setminus\mathbb{Q}$. It follows that:
\begin{enumerate}
  \item If $\varphi=\mathcal{R}_{\varrho}$ is the rotation, then $\Gamma$ is bounded and superficial.
  \item If  $\Gamma$ is bounded, then it is also superficial.
\end{enumerate}
\end{proposition}
\begin{proof} The proposition follows from Proposition \ref{rotation}, Proposition \ref{boundedsuperficial}, Weyl Equidistribution Theorem (asserting that the sequence $(n\varrho)_{n=0}^{\infty}$ for $\varrho \in \mathbb{R}\setminus\mathbb{Q}$ is equidistributed modulo 1) and the fact that $\varphi$ has no periodic orbits.  
\end{proof}

Now let us discuss the case of rational rotation number:
\begin{proposition}\label{sprz_obrot}
Let $\Gamma(u)$ be generated by $u_n=\Phi^{n}(x_0)$, where $\Phi$ is a lift of a circle homeomorphism $\varphi$ with $\varrho=p/q$ ($p$ and $q$ relatively prime), conjugated to the rational rotation $\mathcal{R}_{\varrho}$.

Then $\Gamma(u)$ is not superficial, independently of the choice of $x_0$, and the following conditions are equivalent:
\begin{enumerate}
  \item $\frac{1}{q}\sum_{k=0}^{q-1}\exp{(2\pi\imath \Phi^{k}(x_0))}=0 \in \mathbb{C}$, \label{Prop2_9it1}
  \item $\Gamma(u)$ is bounded, \label{Prop2_9it2}
  \item $\Gamma(u)$ is an equilateral $q$-polygon. \label{Prop2_9it3}
\end{enumerate}
Moreover, $\Gamma(u)$ is a regular polygon for every $x_0\in\mathbb{R}$  if and only if $\varphi=\mathcal{R}_{\varrho}$. In this case and with $x_0=0 \mod 1$, the points $z_0$, $z_1$, $z_q$, $z_{q+1}$, $z_{2q}$, $z_{2q+1}$, ..., $z_{nq}$, $z_{nq+1}$, ... lie on the line $Im(z)=0$.
\end{proposition}
Thus the curves generated by homeomorphisms conjugated to rational rotations, in contrast to those generated by pure rational rotations, can be unbounded and in case they are bounded, they might be equilateral but not regular polygons (i.e. not equiangular). Of course, they can be convex as well as not convex.

\noindent{\textit{Proof of Proposition \ref{sprz_obrot}.}} 
We remark that $\Gamma(\{\Phi^n(x_0)\})$  is bounded if and only if 
\[\sup_{n\in\mathbb{N}}\vert \sum_{k=0}^{n-1}\exp(2\pi\imath \Phi^{k}(x_0))\vert<\infty.\] 
Firstly, we will prove the equivalence of conditions \eqref{Prop2_9it1}-\eqref{Prop2_9it3}. 

Suppose that \eqref{Prop2_9it1} is satisfied, which in this case is equivalent to 
\[
\lim_{N\to\infty}\frac{1}{N}\sum_{k=0}^{N-1}\exp{(2\pi\imath\Phi^k(x_0))}=0.\]
Assume that, on the contrary, $\Gamma(u)$ is not bounded. In particular, this implies that $z_{q}\neq z_0$ because otherwise $\Gamma(u)$ would be a closed curve. So let $c=\vert z_q-z_0\vert$, where $c>0$. Then by periodicity of the orbit $\{\varphi^n(\exp{(2\pi\imath x_0)})\}$, we obtain $\vert z_{2q}-z_0\vert=2c$ and inductively, $\vert z_{nq}-z_0\vert=n c$. But then $\lim_{n\to\infty}\frac{1}{nq}\vert z_{nq}\vert=\frac{c}{q}$ which contradicts \eqref{Prop2_9it1}. On the other hand, if $\Gamma$ is bounded then its Birkhoff average must vanish which means that \eqref{Prop2_9it1} holds.These arguments give equivalence of \eqref{Prop2_9it1} and \eqref{Prop2_9it2}. 

Now assume that \eqref{Prop2_9it2} is satisfied. By periodicity of the orbit of $\exp{(2\pi\imath x_0)}\in S^1$ this means that $z_q=z_0$ since otherwise $\Gamma$ would grow unbounded in the direction of $v=z_q-z_0$. But if $z_q=z_0$ then $\Gamma$ must be an equilateral polygon with $q$ sides (the fact that the sides of this polygon must be of equal length is simply due to the fact that they are vectors of length 1 by definition of a curlicue) and we obtain that \eqref{Prop2_9it2}$\implies$\eqref{Prop2_9it3}. The case \eqref{Prop2_9it3}$\implies$ \eqref{Prop2_9it2} is trivial.

We already know that if $\varphi=\mathcal{R}_{\varrho}$ with $\varrho=p/q$, then $\Gamma(u)$ is a regular polygon with $q$-sides for every $x_0$.  On the other hand, if $\Gamma(u)$ is a regular polygon with $q$ sides for every $x_0$, then all the displacements $\Phi^k(x_0)-\Phi^{k-1}(x_0) \mod 1$ for every $k\in\mathbb{N}$ and $x_0\in\mathbb{R}$ must be equal to $p/q$ which means that $\varphi$ is a rigid rotation. Since then for $x_0=0 \mod 1$ we have $y_{n}=\sum_{k=0}^{n-1}\sin(2\pi \frac{kp}{q})$, where $z_n=(x_n,y_n)$, the last statement follows easily. 

It remains to show non-superficality of $\Gamma(u)$.  For bounded case there is nothing to prove. Similarly, if $\Gamma$ is unbounded then we check the condition $\lim_{t\to\infty}\frac{t}{\textrm{Diam}\Gamma_t}=\infty$. By choosing the subsequence $t_n=nq$ we obtain that $\textrm{Diam}\Gamma_{t_n} \geq \vert z_{nq}-z_0\vert =n c$, where $c=\vert z_{q}-z_0\vert$ and consequently $\lim_{n\to\infty}\frac{t_n}{\textrm{Diam}\Gamma_{t_n}}\leq q/c<\infty$. \qed

\vskip 0.3cm

\begin{example}\label{Example 1}
Let $\varphi$ be conjugated to a rational rotation, i.e. $\Phi=h^{-1}\circ \mathcal{R}_{\varrho} \circ h$, where $\varrho\in\mathbb{Q}$ and $h$ is a lift of some other orientation preserving circle homeomorphisms. For example, define:
\begin{displaymath}
h(x)=\left\{
       \begin{array}{ll}
         \frac{2}{3}x, & \hbox{$0\leq x\leq \frac{3}{8}$;} \\
         2(x-\frac{1}{4}), & \hbox{$\frac{3}{8}\leq x\leq  \frac{1}{2}$;} \\
         \frac{2}{3}(x+\frac{1}{4}), & \hbox{$\frac{1}{2}\leq x\leq \frac{7}{8}$;} \\
         2(x-\frac{1}{2}), & \hbox{$\frac{7}{8}\leq x\leq 1$.}
       \end{array}
     \right.
\end{displaymath}

By applying the rule $h(x+1)=h(x)+1$ (similarly, for $h^{-1}(x)$) we extend $h$ onto an orientation preserving homeomorphism of $\mathbb{R}$.

Let then $\varrho=1/4$. For an arbitrary choice of $x_0$ the orbit $\{\varphi(\hat{x}_0)\}_{n=0}^{\infty}$, $\hat{x}_0=\exp{(2\pi\imath x_0)}$, is periodic with period $4$. In particular, we compute that $\Phi(0)=3/8$, $\Phi^2(0)=1/2$, $\Phi^3(0)=7/8$, $\Phi^4(0)=1 = 0 \mod 1$ and $\Phi^5(0) \mod 1=\Phi(0)$ etc. Thus the displacements $\Phi^{k}(0)-\Phi^{k-1}(0)$ along the trajectory are not equal but, as we easily verify, their average vanishes:
\begin{displaymath}
\frac{1}{4}\sum_{k=0}^3\exp{(2\pi\imath \Phi^k(0))}=0\in\mathbb{C}
\end{displaymath}
According to the above proposition the curve $\Gamma$, evaluated over $\{\Phi^n(0)\}$, is an equilateral polygon, but not regular: it is closed as the average is $0$ but it is not regular since the displacements are not all equal. Indeed, the displacements are alternatingly equal to $3/8$ and $1/8$ and, as we see in Figure \ref{fig:rownoleglobok} (left), $\Gamma$ is a rhombus but not a square. 
\end{example}

\begin{example}\label{Example 2}
Let $x_0=0$ and $h$ be as in Example \ref{Example 1} but take $\varrho=1/5$. In this case the orbit of $x_0$ is obviously periodic (modulo 1) with period 5 but the exponential average along the orbit does not vanish: $\frac{1}{5}\sum_{k=0}^4\exp{(2\pi\imath \Phi^k(0))}\approx (-0.0078,-0.0273)$. Thus $\Gamma$ is unbounded, as reflected in Figure  \ref{fig:rownoleglobok} (right).

\begin{figure}[h!]
\centering
                \includegraphics[width=0.475\textwidth]{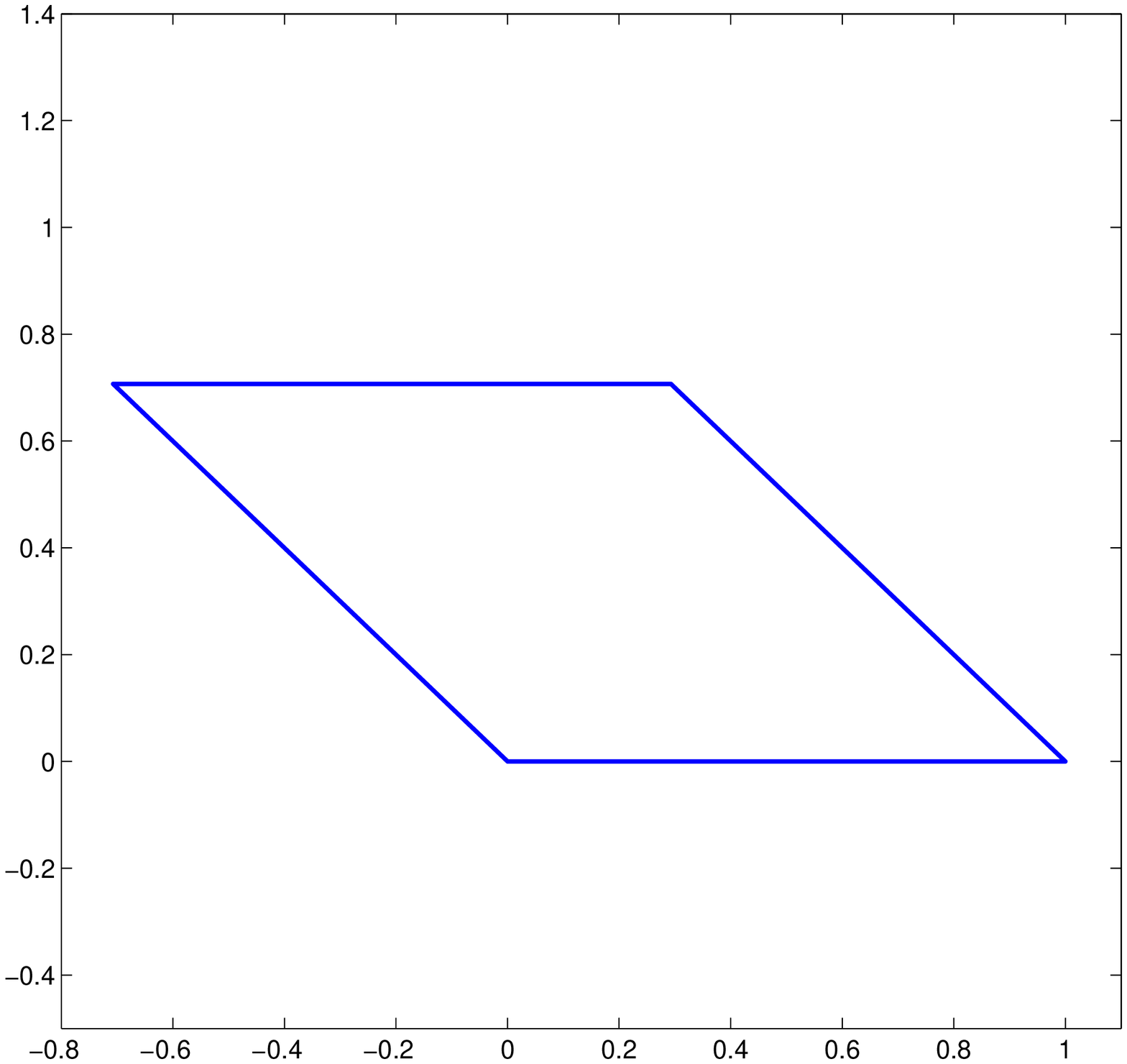}
                \hspace{0.3cm} \includegraphics[width=0.475\textwidth]{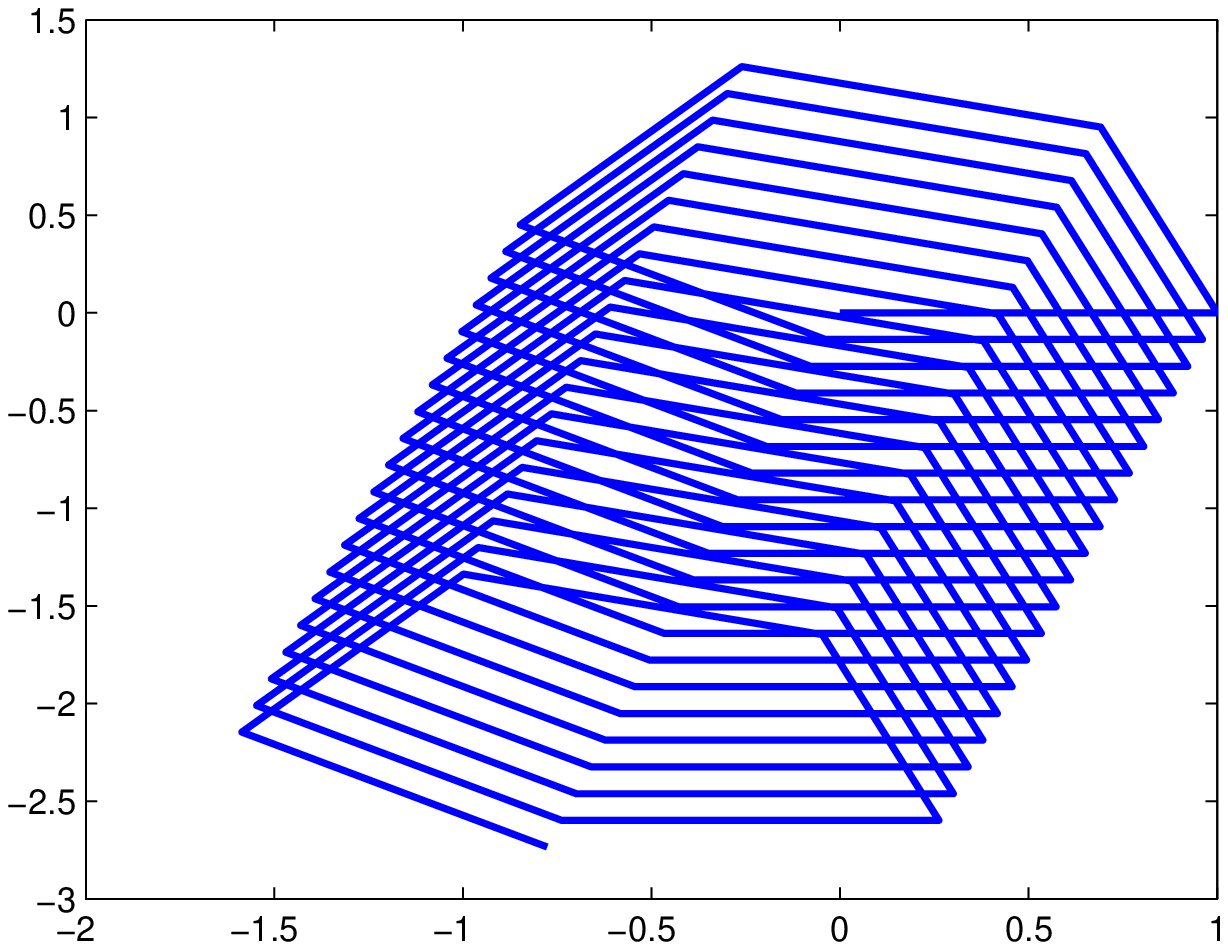}
             \caption{$\Gamma_{100}(0)$ generated by $\varphi$ conjugated with rational rotation as in Example \ref{Example 1} ($\varrho=1/4$, left) and 2 ($\varrho=1/5$, right).}
        \label{fig:rownoleglobok}
\end{figure}

Finally, let us also remark that the boundedness of the curve in Proposition \ref{sprz_obrot} might depend on $x_0$. Indeed, consider for example the lift $\Phi=h^{-1}\circ \mathcal{R}_{1/2}\circ h$, where $h(x)= - 2(x-1/2)^2+1/2$ for $0\leq x\leq 1/2$ and $h(x)= 2(x-1/2)^2+1/2$ for $1/2\leq x\leq 1$ and $\mathcal{R}_{1/2}(x)=x+1/2$ is a lift of a rotation by $\pi$. Let $\Gamma(x)$ denote the curve generated by the orbit $\{\Phi^n(x)\}$. Compare $\Gamma(x_0)$ and $\Gamma(y_0)$ where $x_0=0$, $y_0\in (0,1/2)$. Then $\Gamma(x_0)$ is bounded whereas $\Gamma(y_0)$ is not.  
\end{example}

Now we move to the case of a so-called semi-periodic homeomorphism, which is a homeomorphism with rational rotation number but not conjugated to the rotation,  i.e. when apart from periodic orbits we also have some non-periodic ones. 
\begin{proposition}
Suppose that $\Gamma=\Gamma(x_0)$ is a curve generated by the lift of the orbit of $x_0$ of a semi-periodic circle homeomorphism $\varphi$ with  $\varrho=\frac{p}{q}$.
Then:
\begin{itemize}
  \item If $x_0$ is a periodic point, then the equivalence of conditions \eqref{Prop2_9it1}-\eqref{Prop2_9it3} in Proposition \ref{sprz_obrot} also applies to $\Gamma$. Similarly, $\Gamma$ (bounded or not) is not superficial and it is a regular polygon if and only if the displacements along the periodic orbit $\{\varphi^n(\exp{(2\pi\imath x_0)})\}$ are all equal.
  \item If $x_0$ is not a periodic point and  $\Gamma$ is bounded, then $\Gamma$ is also superficial.
\end{itemize}
\end{proposition}
\begin{proof} The first statement can be proved exactly as Proposition \ref{sprz_obrot}. As for the second statement, when $x_0$ is not a periodic point, then its orbit is attracted  by some periodic orbit of $\varphi$ (see e.g. \cite{katok}) but infinitely many (all) $u_n$'s are different modulo 1 and thus, again on the account of Proposition \ref{boundedsuperficial}, bounded $\Gamma$ is also superficial. \end{proof} 

In particular we realize that periodic orbits of circle homeomorphisms may give rise to unbounded curves. 

\section{Connection with the cohomological equation}\label{secConn}
In this part we are going to show how  the shape of a bounded curlicue generated by a minimal (transitive) circle homeomorphism $\varphi$ is related to the solution of the induced cohomological equation. Till the end of this part we assume that $\varphi\in\mathbb{R}\setminus\mathbb{Q}$. 

It is clear that when the  Birkhoff sums are bounded, i.e. 
\[
\sup_{n\in\mathbb{N}}\vert \sum_{k=0}^{n-1}\exp(2\pi\imath u_k)\vert<\infty, 
\]
then the curve $\Gamma$ is bounded as well. On the other hand, when the  Birkhoff sums are unbounded and the Birkhoff average does not vanish, i.e. 
\[
\lim_{n\to\infty}\frac{1}{n}\vert \sum_{k=0}^{n-1}\exp(2\pi\imath u_k)\vert\neq0,\]
then the curlicue $\Gamma(u)$ is unbounded and grows in the direction of the nonzero vector $\lim_{n\to\infty}\frac{1}{n}\sum_{k=0}^{n-1}\exp(2\pi\imath u_k)=(c_1,c_2)\in \mathbb{C}$. 

In order to verify how consecutive points $z_n$ are located in the plane $\mathbb{C}$ let us recall the classical (Theorem \ref{classicalDK}) and improved (Theorem \ref{improved}) version of the Denjoy-Koksma inequality, using the continued fraction expansion of $\varrho$ ($\varrho\in [0,1)\cap\mathbb{R}\setminus\mathbb{Q}$):
\begin{displaymath}
\varrho=\frac{1}{a_0+\frac{1}{a_1+\frac{1}{\ldots}}}=:[a_0,a_1,a_2,a_3, \ldots],
\end{displaymath}
where $a_0,a_1,...\in \mathbb{Z}$ and 
\begin{displaymath}
\frac{p_n}{q_n}:=\frac{1}{a_0+\frac{1}{a_1+\frac{1}{\frac{\ddots}{a_{n-1}+\frac{1}{a_n}}}}}=:[a_0,a_1,\ldots, a_n]
\end{displaymath}


In all the forthcoming theorems and propositions $q_n$ denotes a denominator of a rational approximation of $\varrho$ by the continued fraction expansion.
\begin{theorem}[cf.\cite{herman}]\label{classicalDK} Let $\varphi:S^1\to S^1$ be a homeomorphism with irrational rotation number $\varrho$ and $g:S^1\to\mathbb{R}$ a real function (not necessary continuous) with bounded variation $\mathrm{Var}(g)$. Then
\begin{equation}\label{nierDK}
\vert \sum_{i=0}^{q_n-1}g(\varphi^i(x))-q_n\int_{S^1}g\;d\mu\vert\leq \mathrm{Var}(g), \quad \forall_{x\in S^1},
\end{equation}
 where $\mu$ is the only invariant Borel probability measure of $\varphi$.
\end{theorem} 

\begin{theorem}[Corollary  in \cite{navas}]\label{improved}
Under the assumptions and notation of Theorem \ref{classicalDK}, if $\varphi$ is a  $C^{1+bv}$ circle diffeomorphism and $g$ is $C^1$, it holds that
\begin{equation}\label{impr}
\|\sum_{i=0}^{q_n-1}g(\varphi^i)-q_n\int_{S^1}g\;d\mu\|_{C^0}\to 0, \quad \textrm{as} \ n\to\infty.
\end{equation}
\end{theorem}

\begin{proposition}\label{uwaga1}
Let $\varphi$ be a circle homeomorphism with irrational rotation number $\varrho$ and a lift $\Phi:\mathbb{R}\to\mathbb{R}$. Suppose that the Birkhoff average equals $\lim_{N\to\infty}\frac{1}{N}\sum_{k=0}^{N-1}\exp(2\pi\imath \Phi^k(x))=c=(c_1,c_2)\in \mathbb{C}$ (allowing also for $(c_1,c_2)=(0,0))$.

Then there exists a constant $M$ such that for every $x\in\mathbb{R}$ we have
\begin{equation}\label{nowe_rn1}
\vert z_{q_n}(x)-q_n c\vert\leq M.
\end{equation}
Moreover, if $\varphi$ is $C^{1+bv}$ diffeomorphism then
\begin{equation}\label{nowe_rn2}
\forall_{\varepsilon>0}\ \exists_{N\in\mathbb{N}}\ \forall_{n>N} \ \forall{x\in\mathbb{R}} \ \vert z_{q_n}(x)-q_n c\vert <\varepsilon.
\end{equation}
\end{proposition}

Note that in case of vanishing Birkhoff average ($c=0\in\mathbb{C}$), \eqref{nowe_rn1} asserts that there is a bounded neighbourhood $U$ of $0\in\mathbb{C}$ such that for every $x_0\in\mathbb{R}$ the points $z_{q_n}(x_0)\in \Gamma(x_0)$, corresponding to the closest-return times of $\varphi$,  return to  this neighbourhood: $z_{q_n}(x_0)\in U \ \ \textrm{for} \  n=0,1,... .$ Simultaneously, by \eqref{nowe_rn2}, for sufficiently large $n$ the points $z_{q_n}(x_0)$ of the curlicue $\Gamma(x_0)$  fail into arbitrarily small neighbourhood of $0$, and this convergence is uniform with respect to $x_0$, provided that $\varphi$ is smooth enough.

On the other hand, if the Birkhoff average does not vanish, then  the curlicue visits neighbourhoods of some points on the straight line in the direction of the non-zero vector $c=(c_1,c_2)$ in the complex plane.

\noindent{\textit{Proof of Proposition \ref{uwaga1}.}} If we consider
 \begin{displaymath}
 z_n(x_0): =\sum_{k=0}^{n-1}\exp(2\pi\imath \Phi^k(x_0))=\sum_{k=0}^{n-1}(\cos(2\pi \Phi^k(x_0))+\imath\sin(2\pi \Phi^k(x_0)))
 \end{displaymath}
  then $g:=g^{(1)}+g^{(2)}$ with $g^{(1)}(x):=\cos(2\pi x)$ and $g^{(2)}(x):=\sin(2\pi x)$ and we can apply the inequality (\ref{nierDK}), respectively to $g^{(1)}$ and $g^{(2)}$.  Note that $z_{q_n}(x_0)=\sum_{i=0}^{q_n-1}g(\varphi^i(x))$. Assume further that 
  \[ 
  \lim_{N\to\infty}\frac{1}{N}\sum_{k=0}^{N-1}\exp(2\pi\imath \Phi^k(x_0))=\int_{[0,1]}\exp(2\pi\imath x)\;d\hat{\mu}(x)=c\in\mathbb{C}, 
  \]
  where $\hat{\mu}$ is the measure $\mu$ lifted to $[0,1]$. Then
  \begin{displaymath}
  \int_{[0,1]}g^{(1)}(x)\;d\hat{\mu}(x)=c_1 \quad \textrm{and} \quad \int_{[0,1]}g^{(2)}(x)\;d\hat{\mu}(x)=c_2.
  \end{displaymath}
Now, as $\mathrm{Var}(g^{(1)})=\mathrm{Var}(g^{(2)})=4$ for $z_n(x_0)=a_n+\imath b_n$ we obtain that $\vert a_{q_n} - q_n c_1\vert\leq 4$ and $\vert b_{q_n}- q_n c_2\vert\leq 4$.   

Similarly, the second statement  follows from Theorem \ref{improved}.  \qed
\vskip 0.3cm

However, in case of vanishing Birkhoff average, the Denjoy-Koksma inequality does not explain in fact whether the curlicue is bounded or not. Nonetheless, this can be achieved by considering the so-called cohomological equation. Let us start from recalling the famous Gottschalk-Hedlund Theorem:
\begin{theorem}[cf. \cite{gh}]\label{twgh}
Let $X$ be a compact metric space and $T:X\to X$ a minimal homeomorphism. Given a continuous function $g:X\to\mathbb{R}$ there exists a continuous function $u:X\to\mathbb{R}$ such that
\begin{displaymath}
g=u-u\circ T
\end{displaymath}
if and only if there exists $K<\infty$ such that
\begin{equation}\label{warunekgh}
\sup_{n\in\mathbb{N}} \vert \sum_{k=0}^{n-1}g(T^k(x_0))\vert<K \ \textrm{for some} \ x_0\in X
\end{equation}
\end{theorem}
Note that every two continuous solutions of the cohomological equation for a minimal homeomorphism $\varphi$ of the compact metric space $X$ differ by a constant (i.e. if $u$ is such a solution, then $\hat{u}:=u+c$, where $c$ is arbitrary constant, is also a solution). Moreover, by minimality of $\varphi$, from \eqref{warunekgh} follows that $\sup_{x\in X}\sup_{n\in\mathbb{N}} \vert \sum_{k=0}^{n-1}g(T^k(x))\vert<2K$ (cf. e.g. \cite{katok}).

In our setting, $X=S^1$. We assume till the end of this section that $\varphi:S^1\to S^1$ is minimal, that is,  $\varphi$ is conjugate to an irrational rotation $\mathcal{R}_{\varrho}$. Therefore $\varphi$ satisfies the assumptions of Theorem \ref{twgh}. We consider $S^1$ as the quotient space $S^1=\mathbb{R}/\mathbb{Z}$ (equivalently, as the interval $[0,1]$ with endpoints identified). Denote by $g:S^1\to\mathbb{C}$ the exponential function $g(x)=\exp(2\pi\imath x)$ and identify $\varphi:S^1 \to S^1$ with its lift $\Phi$ by $\varphi= \Phi \mod 1$. The following functional equation
\begin{equation}\label{ce}
\exp(2\pi\imath x)+u(x)=u(\varphi(x)), \quad \forall_{x\in\mathbb{R}},
\end{equation}
where $u:S^1\to\mathbb{C}$ is the unknown of the problem, will be referred to as the cohomological equation in our further considerations. Traditionally, the cohomological equation is given with respect to the functions $g$ and $u$ taking real values but in our case one can equivalently consider $u=(u_1,u_2)$ and $g=(g_1,g_2)$ with $u_i,g_i: S^1\to\mathbb{R}$, $i=1,2$.  The condition  on bounded Birkhoff sums now takes the form:
\begin{equation}\label{ogr}
\sup_{x\in X}\sup_{n\in\mathbb{N}} \vert \sum_{k=0}^{n-1}\exp(2\pi\imath \varphi^k(x))\vert<K.
\end{equation}
Notice that if  there exists a continuous solution $u$ of the equation (\ref{ce}), then by integrating both sides with respect to the invariant measure $\mu$ of $\varphi$ we obtain that $\int_{S^1}\exp(2\pi\imath x)\;d\mu=0$. In other words, vanishing of the Birkhoff averages $\lim_{N\to\infty}\frac{1}{N}\sum_{k=0}^{N-1}\exp(2\pi\imath \varphi^k(x)))=0$ is a necessary condition for the existence of a continuous solution of (\ref{ce}). 

Let us for a while consider cylinder maps (see e.g. \cite{atkinson}): If $\varphi:S^1\to S^1$ is a minimal homeomorphism, $g:S^1\to\mathbb{C}$ is a continuous function and $Y$ denotes the product space $Y=S^1\times \mathbb{C}$, then the transformation $F:Y\to Y$ given as
\begin{displaymath}
F(x,\xi):=(\varphi(x),\xi+g(x))
\end{displaymath}
is called a \emph{cylinder transformation}. In the current work we consider cylinder transformations of the following form: $F(x,\xi)=(\varphi(x),\xi+\exp(2\pi\imath x))$. Assume that $u(x)$ is the solution of the cohomological equation (\ref{ce}). In this case
\begin{displaymath}
F(x,u(x))=(\varphi(x),u(x)+\exp{2\pi\imath x})=(\varphi(x),u(\varphi(x))).
\end{displaymath}
It follows that $\mathcal{S}:=\{(x,y)\in Y: \ x\in S^1, y=u(x)\}$ is an invariant section of $F$, i.e. $F(\mathcal{S})\subseteq\mathcal{S}$ (compare with the proof of Gottschalk-Hedlund Theorem). We are ready to state
 
 \begin{proposition}\label{main}
Let $\varphi:S^1\to S^1$ be a minimal homeomorphism with a lift $\Phi:\mathbb{R}\to\mathbb{R}$. Then the curve $\Gamma$ generated by an arbitrary orbit $\{\Phi^n(x)\}$ is bounded and superficial if and only if the cohomological equation (\ref{ce}) has a continuous solution $u: S^1\to\mathbb{C}$.

 Moreover, if $\Gamma$, evaluated over the orbit of some point $x_0$,  is bounded then $\lim_{N\to\infty}\frac{1}{N}\sum_{k=0}^{N-1}\exp(2\pi\imath \Phi^k(x_0))=0$ and the points $z_n$ of $\Gamma$ lie on the curve $u(S^1)\subset \mathbb{C}$, where $u$ is a continuous solution of (\ref{ce}) satisfying $u(x_0)=0$.
 \end{proposition}

\begin{proof} The first statement follows directly from Theorem \ref{twgh} and Proposition \ref{boundedsuperficial}. The second part of the proposition can be concluded from the proof of Theorem \ref{twgh} but let us present the short reasoning below.

Notice that $F^n(x,\xi)=(\varphi^n(x),\xi+\sum_{k=0}^{n-1}\exp(2\pi\imath \varphi^k(x)))$, as $g(x)=\exp(2\pi\imath x)$ and $\varphi(x)=\Phi(x) \mod 1$ for $x\in S^1=\mathbb{R}/\mathbb{Z}$. Thus if we choose a point  $x_0$ and let $u$ denote the continuous solution of (\ref{ce}) such that $u(x_0)=0$, then by substituting $\xi=u(x_0)=0$ we get $F^n(x_0,0)=F^n(x_0,u(x_0))=(\varphi^n(x_0),\sum_{k=0}^{n-1}\exp(2\pi\imath \varphi^k(x_0)))$. Consequently, $\pi_2(F^n((x_0),u(x_0)))=\linebreak[10] \sum_{k=0}^{n-1}\exp(2\pi\imath \varphi^k(x_0))=z_n$, where $z_n$ is the $n-th$ vertex of the curlicue evaluated over the orbit $\{\varphi^i(x_0)\}$ and $\pi_2$ is the projection onto the second coordinate (onto the complex plane). But $F^n(x_0,u(x_0))=(\varphi^n(x_0),u(\varphi^n(x_0)))$ since we are on the invariant section $\mathcal{S}\subset S^1\times \mathbb{C}$. Thus $z_n\in\pi_2(\mathcal{S})$. Precisely, $z_n=u(\varphi^n(x_0))$ and, as $u:S^1\to\mathbb{C}$ is continuous,  $u(S^1)=\pi_2(\mathcal{S})$ is a bounded closed curve in the complex plane $\mathbb{C}$ with vertices $z_n$ lying on it.  \end{proof}

An interested Reader can find out more about the existence of induced continuous sections in more general setting for example in the work \cite{ponce}  which studies cocycles of isometries over minimal dynamics. 

 \begin{corollary}
Given an orientation preserving minimal circle homeomorphism $\varphi$ and denoting by  $\Gamma(x_0)$ a curve generated by the orbit of $x_0$, either $\Gamma(x_0)$ is bounded for all $x_0\in S^1$ or for every $x_0$ the curve $\Gamma(x_0)$ is unbounded. Moreover, in case $\Gamma(x_0)$ is bounded, its vertices lie on a closed curve $\tau$ ($\tau=u(S^1)$), whose shape does not depend on the choice of the generating point $x_0$. 
\end{corollary}

We consider the following example of the minimal circle homeomorphism $\varphi$:

\begin{example}\label{Example 3}
Choose an irrational rotation number $\varrho\in\mathbb{R}\setminus\mathbb{Q}$ and let $h$ be as in Example \ref{Example 1}. Then $h$ induces a minimal circle homeomorphism $\varphi$ with $\varphi=\Phi \mod 1$ where
\begin{displaymath}
\Phi=h^{-1}\circ \mathcal{R}_{\varrho}\circ h
\end{displaymath}
and $\mathcal{R}_{\varrho}(x)=x+\varrho$. One checks that
\begin{displaymath}
\int_{0}^{1}\cos(2\pi h^{-1}(x))\;dx=\int_{0}^{1}\sin(2\pi h^{-1}(x))\;dx=0
\end{displaymath}
implying that the Birkhoff average $\lim_{n\to\infty}\frac{1}{n}\sum_{k=0}^{n-1}\exp(2\pi\imath \Phi^k(x_0))$ vanishes (for $\varrho\in\mathbb{R}\setminus\mathbb{Q}$). 

We numerically simulated two cases: $\varrho=\ln 2$ and $\varrho=\pi$ and for each of them obtained a bounded curve (suggesting that the corresponding  Birkhoff sums are bounded and the corresponding cohomological equations have continuous solutions). The results are presented in Figure \ref{fig:example}. Note also that these values of $\varrho$ are Diophantine (see Definition \ref{diophDefin}).
\begin{figure}[h!]

                \includegraphics[width=0.47\textwidth]{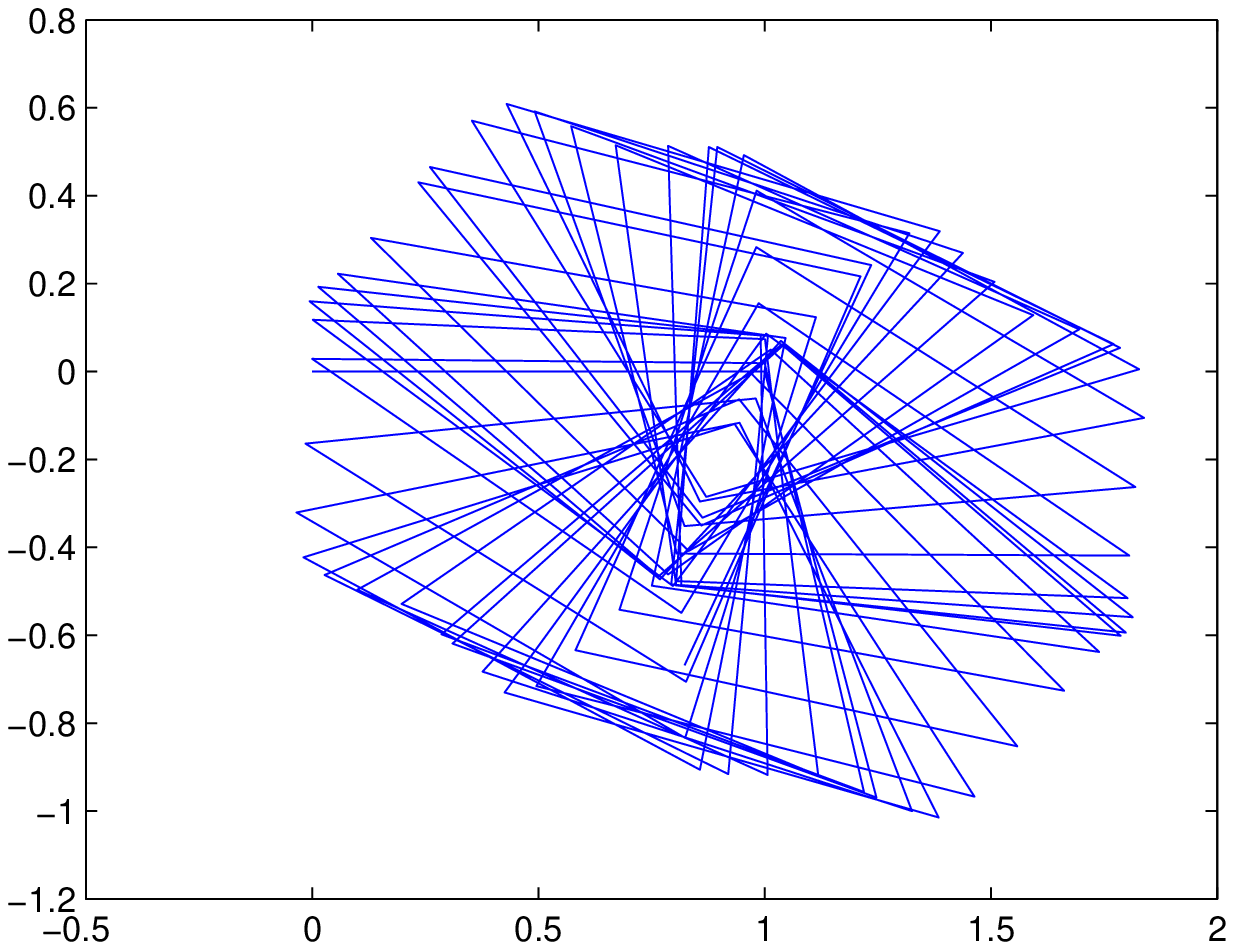}
                \includegraphics[width=0.47\textwidth]{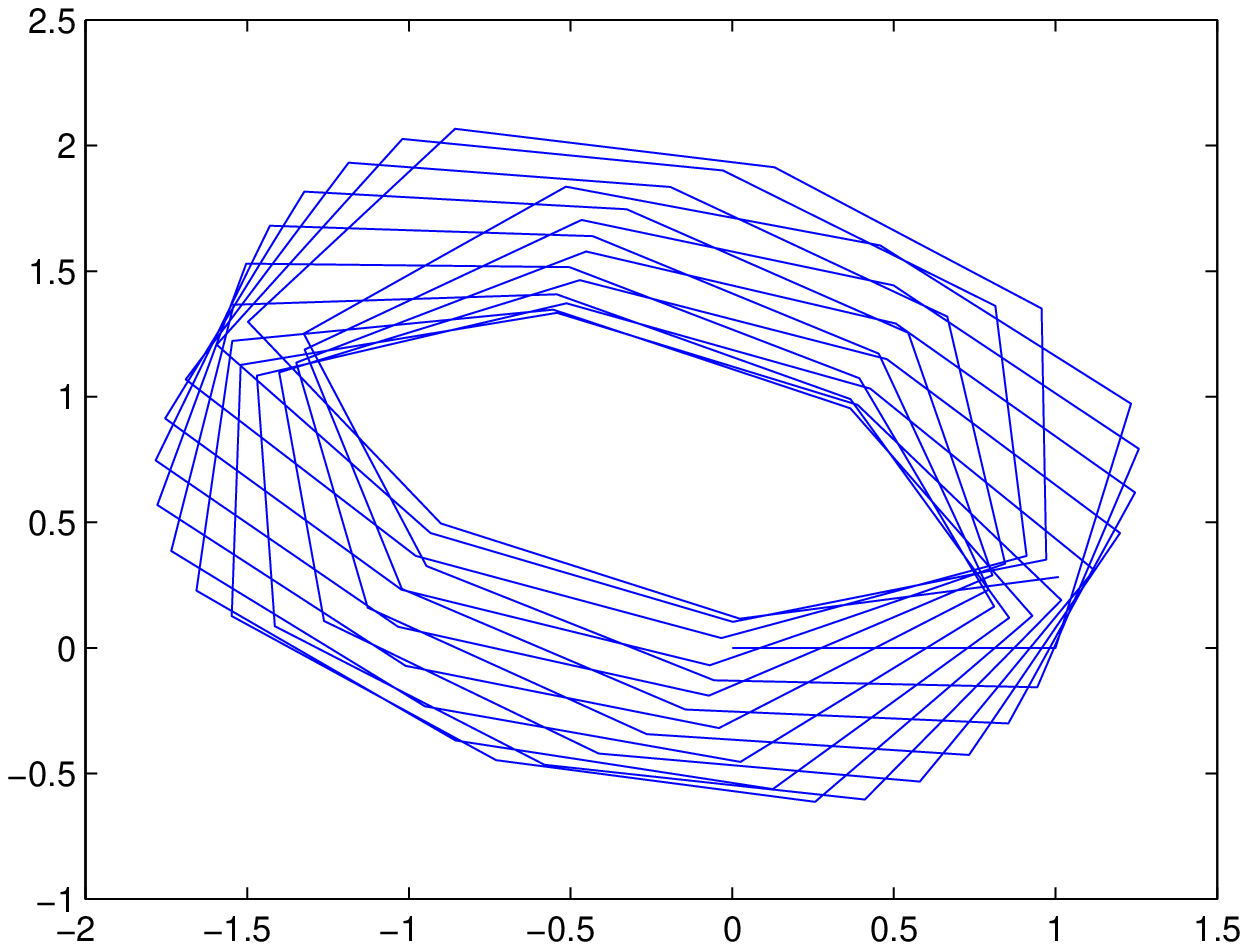}\\

                \includegraphics[width=0.47\textwidth]{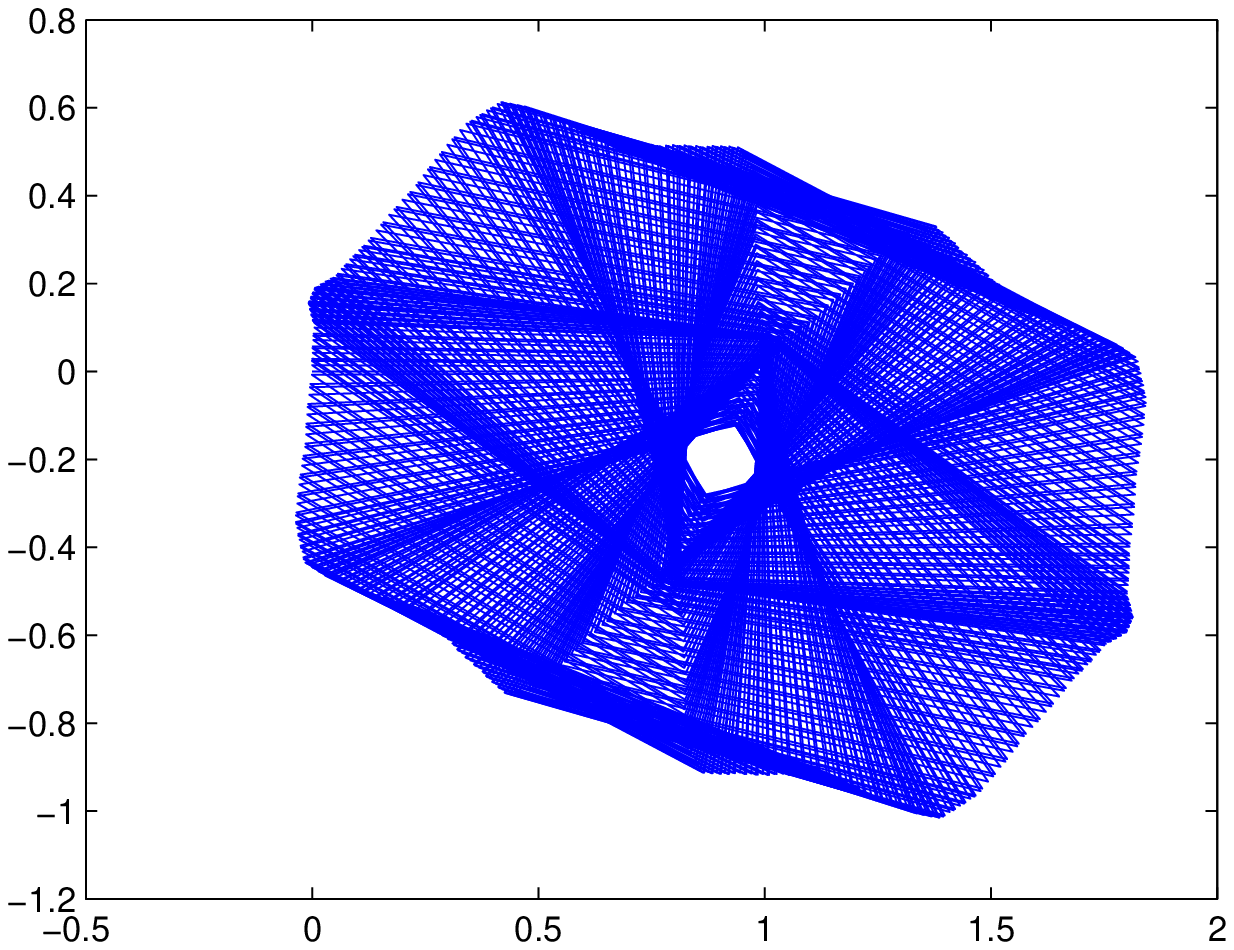}
                 \includegraphics[width=0.47\textwidth]{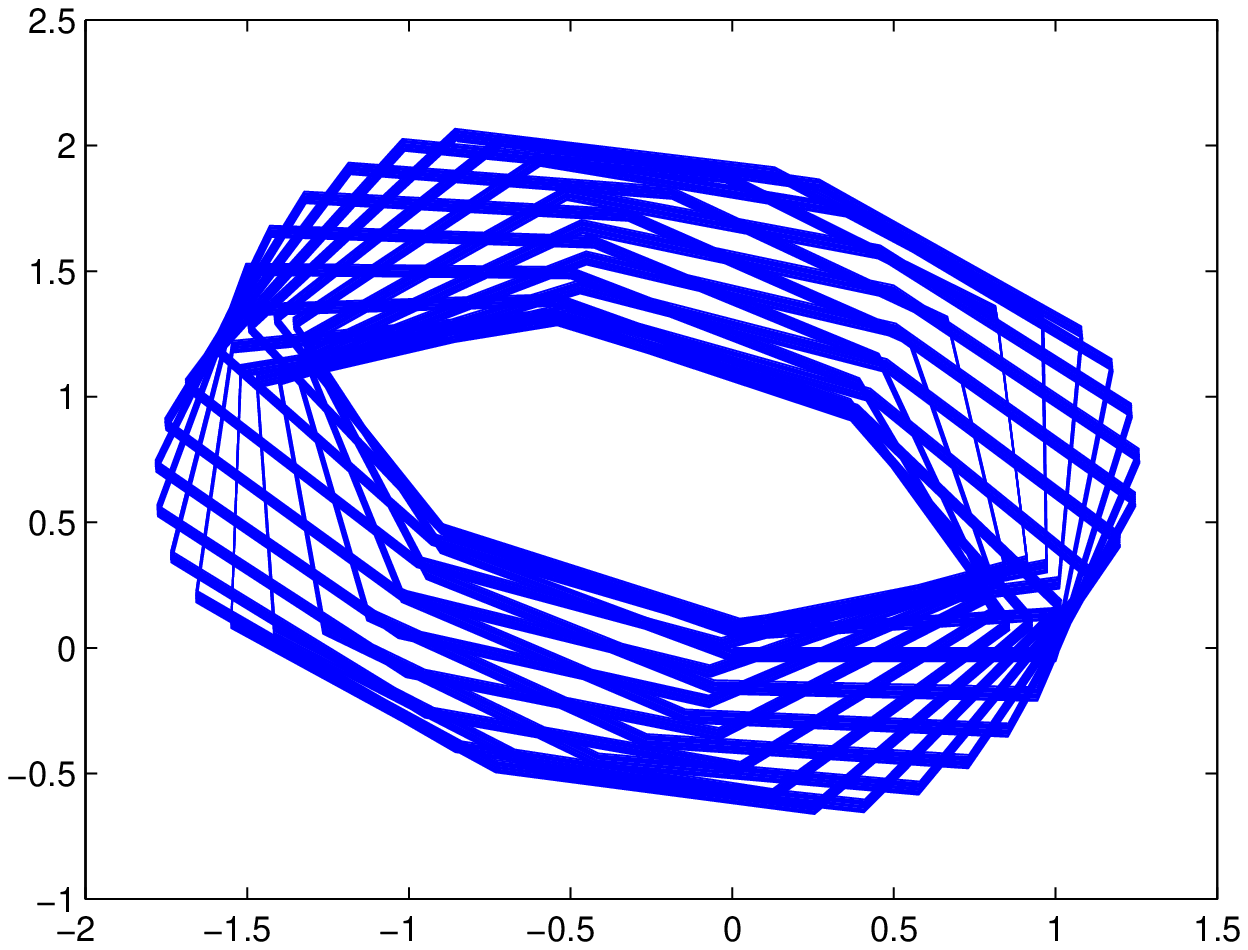}
                               
        \caption{$\Gamma_{n}(x_0)$ generated by $\Phi=h^{-1}\circ \mathcal{R}_{\varrho}\circ h \mod 1$ (see Example \ref{Example 3}) with $x_0=0$ and different values of $\varrho$ (left panel: $\varrho=\ln 2$, $n\in\{100,1000\}$; right panel: $\varrho=\pi$, $n\in\{100,1000\}$)}\label{fig:example}
\end{figure}
\end{example}

\section{Growth rate and superficiality for unbounded curlicues}\label{secGrow}
In this part we deal with curves $\Gamma$ generated by circle homeomorphisms $\varphi=\Phi \mod 1$ with irrational rotation number $\varrho\in\mathbb{R}\setminus\mathbb{Q}$. We already know that if $\Gamma$ is bounded then its shape can be described by a solution of the certain cohomological equation. Notwithstanding, in case $\Gamma$ is unbounded we might always ask how fast $\textrm{Diam}(\Gamma_t)$ increases to $\infty$ and whether $\Gamma$ is superficial.

\begin{definition}\label{diophDefin}
A real number $\varrho$ for which there exists $C>0$ and $r>1$ satisfying
\begin{displaymath}
\vert \varrho-\frac{p}{q}\vert\geq \frac{C}{q^{1+r}}
\end{displaymath}
for all $p/q\in\mathbb{Q}$ is called Diophantine of type $r$.
\end{definition}

We remark that for every fixed $r>1$ the set of real numbers of Diophantine type $r$ has full Lebesgue measure. Moreover, the intersection of sets of Diophantine numbers of type $r$ over all $r>1$ has full measure too.

\begin{definition}
A real number is of \emph{bounded type} if the continued fraction approximation $\frac{p_n}{q_n}$ has the property that $\frac{q_{n+1}}{q_n}$ is bounded.
\end{definition}

\begin{theorem}\label{TwierdzenieZbiorcze}
Assume that $\varrho\in\mathbb{R}\setminus\mathbb{Q}$.

Suppose that $\int_{0}^1\exp(2\pi\imath x)\; d\mu(x)=0$, where $\mu$ is the unique invariant Borel probability measure of $\varphi=\Phi \mod 1$. In this case:
\begin{enumerate}
\item If $\varrho$ is Diophantine of type $r>1$, then for every $x_0\in\mathbb{R}$ 
\begin{equation}\label{estim1}
\vert z_n(x_0)\vert=\vert\sum_{k=0}^{n-1}\exp{(2\pi\imath \Phi^n(x_0))}\vert=\mathcal{O}(n^{1-1/r}\log(n)).
\end{equation} \label{thm4_3id1}
\item  If $\varrho$ is of bounded type, then for every $x_0\in\mathbb{R}$ we have
\begin{equation}\label{estim2}
\vert z_n(x_0)\vert=\vert\sum_{k=0}^{n-1}\exp{(2\pi\imath \Phi^n(x_0))}\vert=\mathcal{O}(\log(n)).
\end{equation} \label{thm4_3id2}
\item If  $\varrho$ satisfies $a_m<(m+1)^{1+\varepsilon}$, with $\varepsilon>0$,  for any $m$ large enough, where $a_m$'s are the integers appearing in the continued fraction expansion of $\varrho$ ($\varrho=[0;a_0,a_1,a_2,...]$), then above estimates can be replaced with
\begin{equation}\label{estim3}
\vert z_n(x_0)\vert=\vert \sum_{k=0}^{n-1}\exp{(2\pi\imath x)}\vert=\mathcal{O}(\log^{2+\varepsilon}n).
\end{equation}\label{thm4_3id3}
\end{enumerate}
In any of the above cases \eqref{thm4_3id1}-\eqref{thm4_3id3} the curve $\Gamma$ generated by $\varphi$ is superficial.

On the other hand, if the Birkhoff average $\int_{[0,1]}\exp(2\pi\imath x)\;d\mu(x)=v$ does not vanish,  then $\Gamma$ is not superficial. However, if its rotation number $\varrho$ satisfies the hypothesis of any of \eqref{thm4_3id1}-\eqref{thm4_3id3} above,  then there exists a constant $M$ and a function $l(n)=o(n)$ (i.e. $\lim_{n\to\infty}\frac{l(n)}{n}=0$) such that for any $x_0\in\mathbb{R}$ and $n\in\mathbb{N}$ we have
 \begin{equation}\label{estim4}
 \vert z_n(x_0)-n v\vert\leq M l(n).
 \end{equation}
 \end{theorem}
 We remark that the set of irrationals whose partial quotients satisfy the assumptions of \eqref{thm4_3id3} is of full measure.  
 
\begin{proof} The above claims can be concluded from Denjoy-Koksma inequality (\ref{nierDK}) and the assumed additional properties of the rotation number. Alternatively, we refer the Reader to \cite{jetomir} and \cite{knill}, where a similar fact as in \eqref{thm4_3id1} above is shown for the irrational rotation $\varphi(x)=x+ \varrho \mod 1$ and for arbitrary function $g(x)$ with bounded variation $\textrm{Var}(g)$ and such that $\int_0^{1}g(x) dx=0$ (here $g(x)=\exp{(2\pi\imath x)}$ and $M^{\prime}:=\textrm{Var}(\exp{(2\pi\imath x)})=2\pi$). Proofs therein remain valid for any homeomorphism with irrational rotation number (of Diophantine type $r$) provided that the condition $\int_0^{1}g(x) dx=0$ is replaced with $\int_0^{1}g(x) d\mu(x)=0$ since the Denjoy-Koksma inequality holds for an arbitrary such homeomorphism. In fact, concerning \eqref{estim1}, one can prove that there exists a constant $M$ such that for every $x_0\in\mathbb{R}$ and $n\in\mathbb{N}$ we have
\begin{equation}\label{estim1_a}
\vert z_n(x_0)\vert\leq M n^{1-1/r}\log(n).
\end{equation}

In order to prove \eqref{thm4_3id2} it suffices to show that there exists a constant $M$ such that for every $x_0\in\mathbb{R}$ and $n\in\mathbb{N}$ we have
\begin{equation}\label{estim2_a}
\vert z_n(x_0)\vert=\vert\sum_{k=0}^{n-1}\exp{(2\pi\imath \Phi^n(x_0))}\vert\leq M\log(n).
\end{equation}
Notice that (\ref{nierDK}) and $\varrho$-Ostrowski expansion (see \cite{ostrowski}) imply that for arbitrary $n$ there exist $m=m(n)\leq\frac{2\log n}{\log 2}$ and a sequence of integers $c_0, c_1, ..., c_m$ such that
\begin{displaymath}
\vert z_n(x_0)\vert=\vert \sum_{k=0}^{n}\exp{(2\pi\imath \Phi^k(x_0))}\vert\leq (c_0+\ldots c_m)M^{\prime}\leq (\sum_{k=0}^{m}\frac{q_{k+1}}{q_k})M^{\prime},
\end{displaymath}
Now, since $\varrho$ is of bounded type, there exists a constant $C$ such that $q_{i+1}/q_{i}<C$, $i=0,1,...,m$ and thus we have $\vert z_n(x_0)\vert< m M^{\prime}C$ which gives the desired estimate. 

The statement \eqref{thm4_3id3} is a counterpart of the corresponding Proposition 2.3 in \cite{guillotin} on irrational rotations, which can be readily extended to minimal homeomorphisms. Therefore completing the proof for vanishing Birkhoff average only amounts to justifying superficiality.  To this end, note that if $\Gamma$ is bounded then we immediately obtain that it is superficial, since infinitely many ${u_n}$'s are different modulo $1$. Suppose now that $\Gamma$ is unbounded and that \eqref{thm4_3id1}, \eqref{thm4_3id2} or \eqref{thm4_3id3} holds.  Then there exist a constant $M$ and a function $l(\cdot)$ such that $l(n)=o(n)$  and 
 \begin{displaymath}
 \lim_{n\to\infty}\frac{n}{\textrm{Diam}\Gamma_n}\geq \lim_{n\to\infty}\frac{n}{2\vert z_n(x_0)\vert}\geq \lim_{n\to\infty}\frac{n}{M l(n)}=\infty,
 \end{displaymath}
where $\Gamma$ is generated by the orbit of $x_0 \in \mathbb{R}$. It follows that $\Gamma$ is superficial.

It is also clear that the estimate \eqref{estim4} concerning non-vanishing Birkhoff average can be obtained in a similar manner.  In order to prove non-superficiality, let us take the increasing sequence $q_n\to\infty$ of the closest returns. Again for some constant $M$ we have
\begin{displaymath}
q_n\vert v\vert - M\leq \vert z_{q_n}(x_0)\vert\leq M+q_n\vert v\vert
\end{displaymath} and
\begin{displaymath}
\lim_{n\to\infty}\frac{q_n}{\textrm{Diam}\Gamma_{q_n}}\leq \lim_{n\to\infty}\frac{q_n}{\vert z_{q_n}(x_0)\vert}\leq \lim_{n\to\infty}\frac{q_n}{q_n\vert v\vert - M}=\frac{1}{\vert v\vert},
\end{displaymath}
which violates the condition that $\lim_{t\to\infty}\frac{t}{\textrm{Diam}\Gamma_t}=\infty$.\end{proof}

 Let us recall that if $\varrho \in \mathbb{R}\setminus\mathbb{Q}$ and $\Gamma$ is bounded, then it is always superficial. Thus the Theorem \ref{TwierdzenieZbiorcze} addresses mainly (non-)superficiality and growth rate for unbounded curves.  
 
 \section{Local discrete radius of curvature}\label{secLoca}
The last geometric feature we are going to study is the local discrete radius of curvature, which after Sinai (\cite{curl}) we define as:
\begin{definition}
Let $z_n:=\sum_{k=0}^{n-1}\exp(2\pi\imath u_k)$. The radius of the circle which goes through the consecutive points $z_{n-1}$, $z_n$ and $z_{n+1}$ on the curve $\Gamma(u)$ is called \emph{local discrete radius of curvature} and denoted $r_n$.
\end{definition}

Direct calculations justify the following (compare \cite{curl}):
\begin{proposition}
If $\Gamma$ is generated by the sequence $u=(u_n)$, then
\begin{displaymath}
r_n=\frac{1}{2}\vert\mathrm{cosec}(\frac{\eta_n}{2})\vert,
\end{displaymath}
where $\eta_n$ is the displacement between the elements $u_{n-1}$ and $u_n$:
\begin{displaymath}
\eta_n:=2\pi (u_n-u_{n-1}).
\end{displaymath}
\end{proposition}
In particular, if $u_n=\Phi^n(x_0)$ for some $\Phi:\mathbb{R}\to\mathbb{R}$, then
\begin{displaymath}
\eta_n:=2\pi\Psi(\Phi^{n-1}(x_0)),
\end{displaymath}
where $\Psi:=\Phi-\textrm{Id}$ is the displacement function of $\Phi$. Thus when $\Gamma$ is generated by an orientation preserving homeomorphism $\varphi:S^1\to S^1$, the sequence $\{r_n\}$ has exactly the same properties as the displacement sequence of an orientation preserving circle homeomorphism, studied in \cite{wmjs1} and \cite{wmjs2}. In particular we conclude 

\begin{theorem}\label{wlasnosci1} Let $r_n$ be the local discrete radius of curvature of the curve $\Gamma$ generated by the orbit of $x_0$ under the lift $\Phi$ of an orientation preserving circle homeomorphism $\varphi$ with the rotation number $\varrho$. Then

\begin{enumerate}
  \item If $\varphi$ is a rotation $\mathcal{R}_{\varrho}$ by $2\pi \varrho$, then the sequence $r_n$ is constant: $r_n=\frac{1}{2}\vert\mathrm{cosec}(\pi \varrho)\vert$.
  \item If $\varphi$ is conjugated to the rational rotation by $2\pi\varrho$, where $\varrho=\frac{p}{q}$,  then the sequence $r_n$ is $q$-periodic.
  \item For a semi-periodic circle homeomorphism $\varphi$, the sequence
$r_n$ is asymptotically periodic. Precisely, if $\varrho(\varphi)=p/q$ then:
\begin{equation}\label{semiokresowosc}
\forall_{\varepsilon>0}\ \exists_{N\in\mathbb{N}} \ \forall_{n>N}\ \forall_{k\in\mathbb{N}}\ \quad |r_{n+kq}(z)-r_{n}(z)|<\varepsilon
\end{equation}
  \item If $\varrho\in \mathbb{R}\setminus \mathbb{Q}$ and $\varphi$ is minimal  then  the sequence $\{r_n\}$ is \emph{almost strongly recurrent}, i.e.
  \begin{displaymath}
\forall_{\varepsilon>0} \ \exists_{N\in\mathbb{N}} \ \forall_{n\in \mathbb{N}} \ \forall_{k\in\mathbb{N} \cup \{0\}} \ \exists_{i\in\{0,1,..., N\}} \ |r_{n+k+i}-r_{n}|<\varepsilon
\end{displaymath}
and $\{r_n\}$ is dense in the set
\begin{displaymath}
[\min_{x\in [0,1]} g(\Psi(x)), \max_{x\in [0,1]} g(\Psi(x))]=[\min_{x\in [0,1]} g(\Omega(x)), \max_{x\in [0,1]} g(\Omega(x))],
\end{displaymath}
where $g(x):=\frac{1}{2}\vert\mathrm{cosec}(\pi x)\vert$ and $\Omega(x):=G^{-1}(x+\varrho)-G^{-1}(x)$, with $G$ being the lift of a homeomorphism $\gamma$ conjugating $\varphi$ with the corresponding rotation. 
\end{enumerate}
\end{theorem}
We add that one can give a counterpart of this theorem for non-transitive homeomorphisms (see Proposition 2.1 in \cite{wmjs2}).  It is worth noting that the spiral-like components of curves $\Gamma=\Gamma((u_n))$ occur for those $n$ where $r_n$ is close to the minimum of $g(x)=\frac{1}{2}\vert\mathrm{cosec}(\pi x)\vert =1/2$ (see \cite{curl}). Moreover, the repetitive-like structure of the $(r_n)_{n=1}^{\infty}$ sequence induced by circle homeomorphisms, as captured by Theorem \ref{wlasnosci1}, explains visible recurrence of ``similar'' parts of the curlicue as illustrated e.g. in Figure \ref{fig:Arnold}.

\begin{example}\label{Example 4}
We have numerically investigated the family of Arnold circle maps:
\begin{displaymath}
\varphi(x)=x+\omega-\frac{K}{2\pi}\sin(2\pi x) \mod 1
\end{displaymath}
with the results presented in Figure \ref{fig:Arnold}. Taking $x_0=0$ and $K=1$, for $\omega=1/2$ we obtained the horizontal segment  of length 1 (not shown) since in this case $x_0=0$ is a 2-periodic  point of $\varphi$. However, for $\omega=1/3$ and $\omega=1/4$ the curlicue $\Gamma(0)$ accumulated along the straight line with some regular patterns visible after zooming in.

\begin{figure}[h!]
       
                \includegraphics[width=0.47\textwidth]{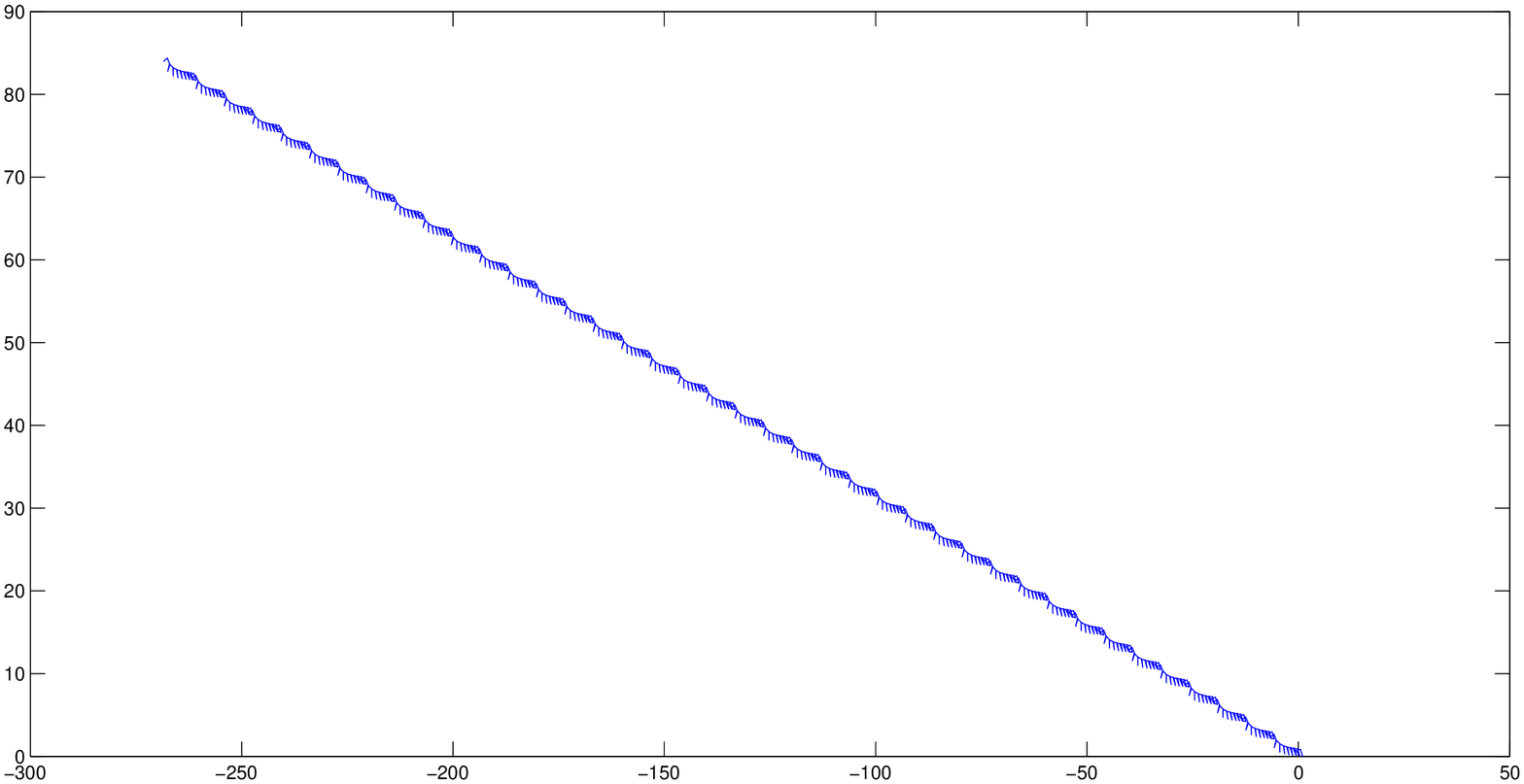}
                \includegraphics[width=0.47\textwidth]{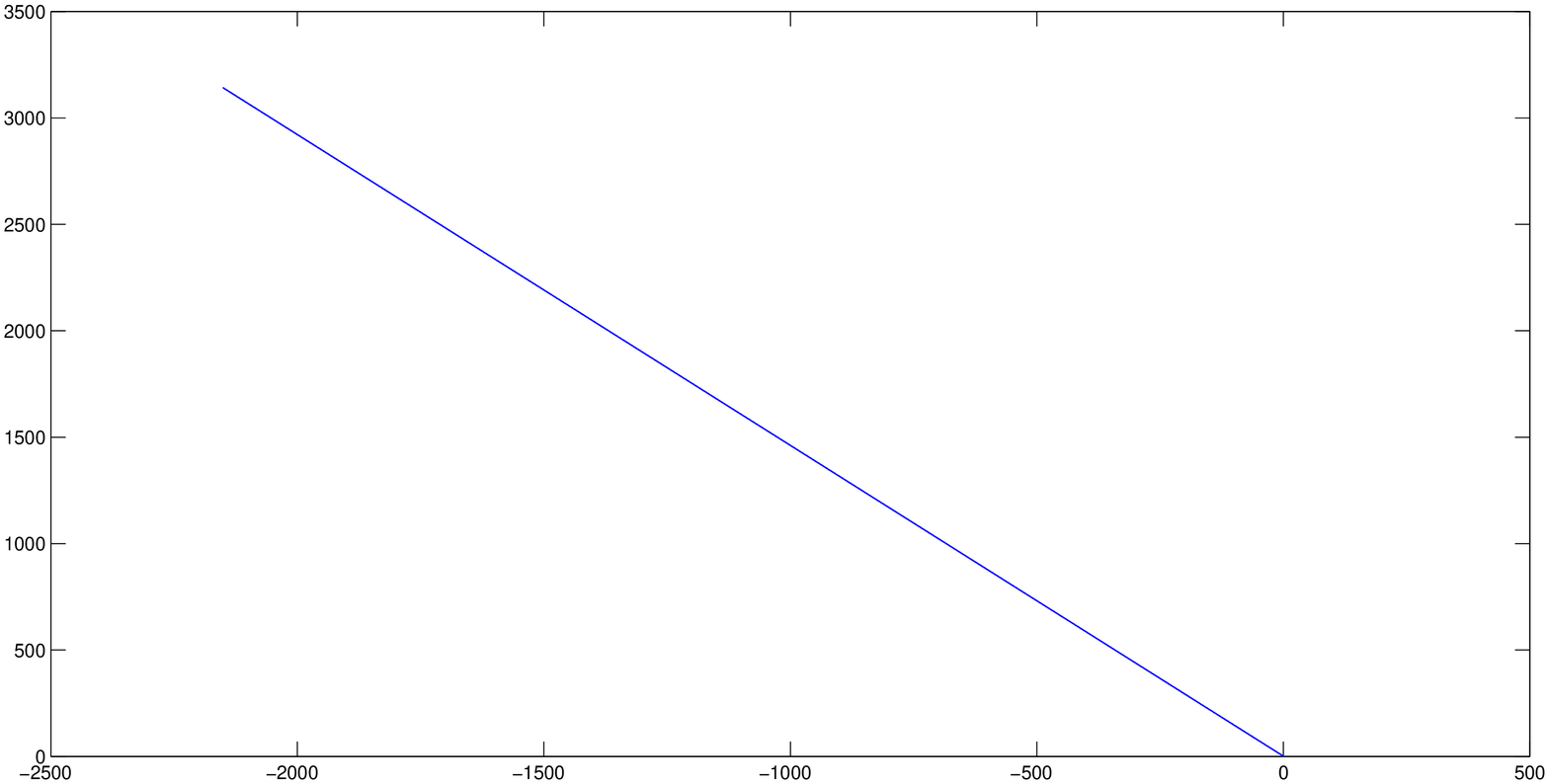}\\

                \includegraphics[width=0.47\textwidth]{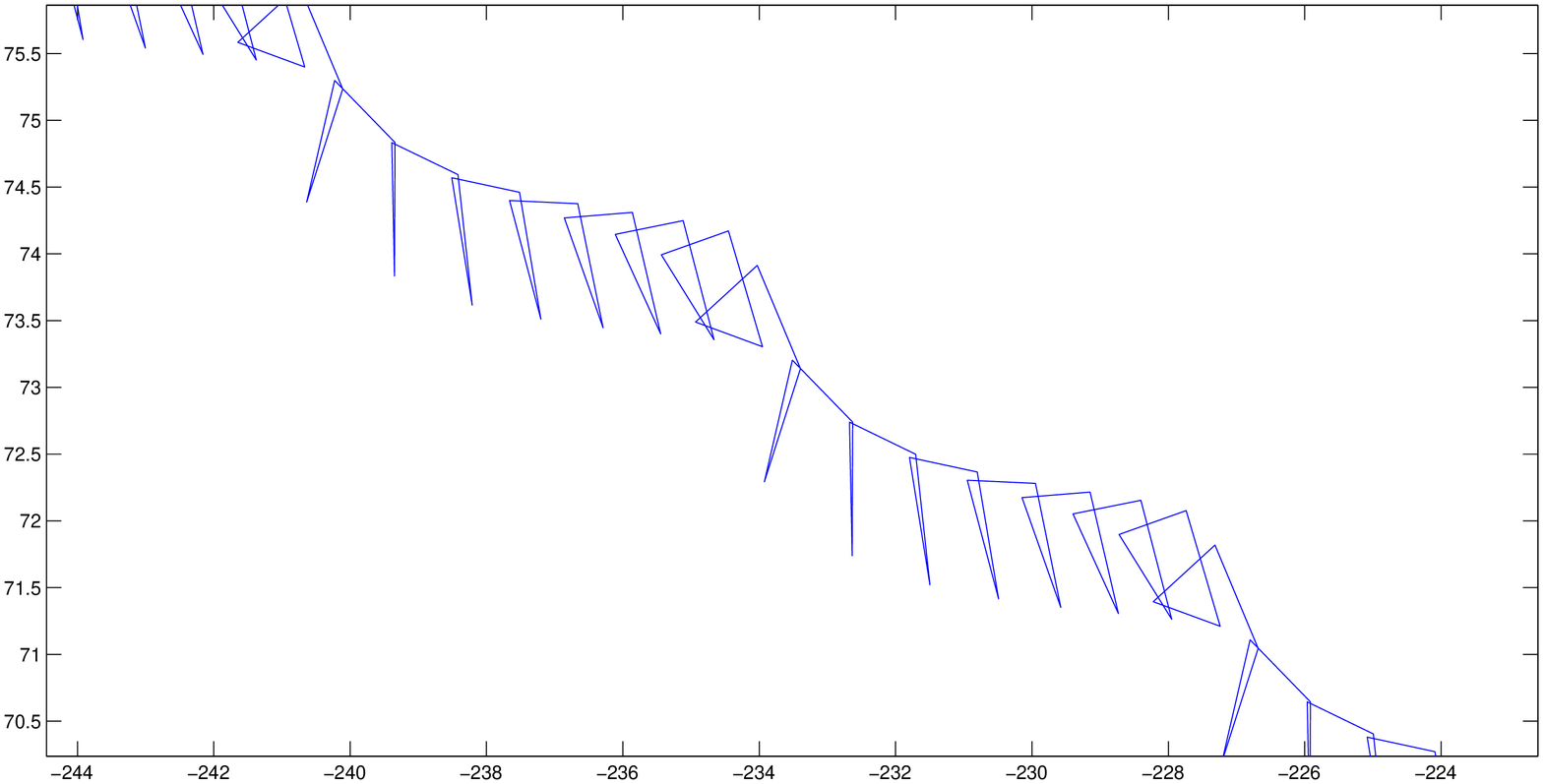}
                 \includegraphics[width=0.47\textwidth]{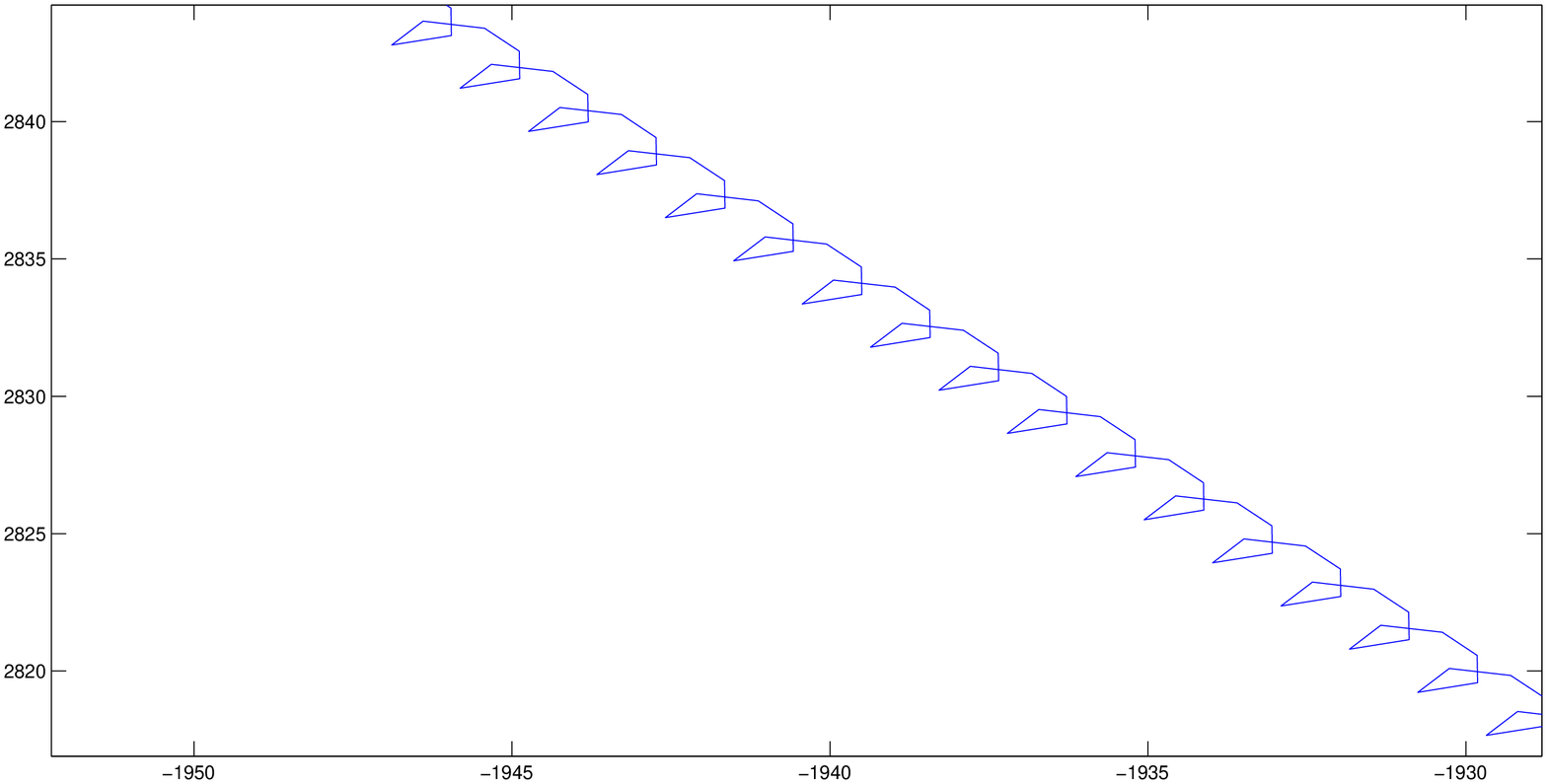}
                
        \caption{$\Gamma_{n}(x_0)$ generated by the Arnold map at $x_0=0$ with $K=1$ and different values of $\omega$ (left panel: $\omega=1/3$, $n=1000$, right panel: $\omega=1/4$, $n=10000$) with visible structure after zooming (bottom)}\label{fig:Arnold}
\end{figure}
\end{example}

We shall also ask about the distribution of the elements of the sequence $\{r_n(x_0)\}$ if $\varphi$ has irrational rotation number. 
\begin{definition}
Let $A\subset \mathbb{R}^+ = [0,+\infty)$ be a Borel subset. We define the distribution $\omega$ of the elements of $\{r_n(x_0)\}$ as
\begin{equation}\label{conv}
\omega(A):=\lim_{n\to\infty}\frac{\vert \{k\in \{1,2,...,n\}: r_k(x_0)\in A\}\vert}{n}
\end{equation}
\end{definition}
From the fact that $\varphi$ is uniquely ergodic we obtain:
\begin{proposition}
If $\varrho\in \mathbb{R}\setminus\mathbb{Q}$, then for every Borel set $A\subset \mathbb{R}$ we have
\begin{displaymath}
\omega(A)=\int\limits_{[0,1]}\mathcal{X}_A\circ F\;d\mu=\mu(\{F^{-1}(A)\}),
\end{displaymath}
where $F(x):=\frac{1}{2}\vert\mathrm{cosec}(\pi\Psi(x))\vert$ and $\mu$ is the unique invariant ergodic measure for the homeomorphism $\varphi = \Phi \mod 1$. In particular, $\omega$ does not depend on the choice of the generating point $x_0$ and the above convergence is uniform with respect to $x_0$.  The \emph{average local discrete radius of curvature} equals:
\begin{displaymath}
\hat{r}:=\lim_{n\to\infty}\frac{1}{n}\sum_{k=1}^{n}r_k=\frac{1}{2}\vert\mathrm{cosec}(\pi\varrho)\vert
\end{displaymath}
\end{proposition}

Let us remark that from Proposition 2.8 and Theorem 2.17 in \cite{wmjs2} one readily obtains a kind of stability (in terms of weak convergence of measures) for (sample-)distributions of the elements of the sequence $r_n$ as $\varphi$ is approximated by some other homeomorphism $\tilde{\varphi}$, close to $\varphi$ in $C^0(S^1)$, allowing also for rational rotation number $\tilde{\varrho}$ of $\tilde{\varphi}$. This can serve as the justification for numerically estimated $\{r_n(x_0)\}$ distributions.

However, one can also ask how the radius $r^{\alpha}_n$ itself (not the distribution) depends on the parameter $\alpha$ when $\{\varphi_\alpha\}$ is a continuously parameterized family of circle homeomorphisms $\varphi_\alpha$:
\begin{proposition}\label{ciagloscpromienia}
Let $\varphi:S^1\to S^1$ be a minimal homeomorphism with an irrational rotation number $\varrho$. Fix $\varepsilon>0$. Then there exists a neighbourhood $U\subset C^{0}(S^1)$ of $\varphi$ such that for every other minimal homeomorphism $\tilde{\varphi}\in U$ with the same rotation number $\varrho(\tilde{\varphi})=\varrho$ we have
\begin{displaymath}
\sup_{x_0\in\mathbb{R}}\sup_{n\in\mathbb{N}}\vert r_n(x_0)-\tilde{r}_n(x_0)\vert <\varepsilon,
\end{displaymath}
where $r_n(x_0)$ and $\tilde{r}_n(x_0)$ denote local radii of curvature evaluated at $x_0$, respectively, for $\varphi$ and $\tilde{\varphi}$.
\end{proposition}
\begin{proof} Firstly note (see Theorem  2.3 in \cite{wmjs2}) that the mapping $\varphi\mapsto \gamma$ assigning to a homeomorphism  with irrational rotation number a map $\gamma: S^1 \to S^1$ semi-conjugating it (or conjugating, if  is minimal)  with the corresponding rotation is a continuous mapping from $C^0(S^1)$ into $C^0(S^1)$-topology (up to some normalization, since every two (semi-)conjugacies of $\varphi$ differ by an additive constant in the lift).

We recall that
\begin{eqnarray}
  r_n(x_0) &=& \frac{1}{2}\vert \mathrm{cosec}(\pi(\Phi^n(x_0)-\Phi^{n-1}(x_0)))\vert \nonumber\\
  \tilde{r}_n(x_0) &=& \frac{1}{2}\vert \mathrm{cosec}(\pi(\tilde{\Phi}^n(x_0)-\tilde{\Phi}^{n-1}(x_0)))\vert, \nonumber
\end{eqnarray}
where $\Phi$ and $\tilde{\Phi}$ are corresponding lifts. Let $\Psi=\Phi-\textrm{Id}$ and $\tilde{\Psi}=\tilde{\Phi}-\textrm{Id}$ denote corresponding displacement functions. Then $\Psi,\tilde{\Psi}: \mathbb{R}\to\mathbb{R}$ are continuous and periodic with period $1$. Moreover, $\Psi(\mathbb{R})=\Psi([0,1])\subset (k,k+1)$ for some $k\in\mathbb{Z}$, as there are no periodic points of $\varphi$ (similarly for $\tilde{\Psi}$). However, the lift $\Phi$ of $\varphi$ can be chosen so that $\Psi(\mathbb{R})\subset (0,1)$ (the shape of the curlicue and the radius of curvature do not depend on the choice of the lift). But as $\Psi$ attains its lower and upper bounds there exists $\delta$ such that $\Psi(\mathbb{R})\subset [\delta,1-\delta]\subset (0,1)$. By considering sufficiently small neighbourhood $U\subset C^0(S^1)$ of $\varphi$ we can assume that $\tilde{\Psi}(\mathbb{R})\subset [\delta,1-\delta]$ for $\tilde{\Psi}$ being the displacement of an arbitrary $\tilde{\varphi}\in U$. Now consider the function $h(x)=\mathrm{cosec}(\pi x)$ on the interval $[\delta,1-\delta]$.  There exists $M$ such that $\vert h^{\prime}(x)\vert<M$ for every $x\in [\delta,1-\delta]$. Fix $\gamma$, which conjugates $\varphi$ with the rotation by $2\pi\varrho$ and let $G$ be its lift. Let us also fix $\omega$ and $\tau<\omega$ which can be  arbitrary small  numbers such that $\vert G^{-1}(x)-G^{-1}(y)\vert<\omega$ whenever $\vert x-y\vert<\tau$. After possibly further decreasing the neighbourhood $U$, we can assume that for every minimal $\tilde{\varphi}\in U$  
there exists $\tilde{\gamma}$ (semi-)conjugating $\tilde{\varphi}$ with its corresponding rotation such that $d_{C^0}(\gamma,\tilde{\gamma})<\tau$ and $d_{C^0}(\gamma^{-1},\tilde{\gamma}^{-1})<\tau$. Thus let us choose  $\tilde{\varphi}\in U$, which is a minimal homeomorphism with the same rotation number $\varrho$. Let $x_0\in\mathbb{R}$ be arbitrary. 
Then $d_{C^0}(\gamma,\tilde{\gamma}),d_{C^0}(\gamma^{-1},\tilde{\gamma}^{-1})<\tau$ for some $\tilde{\gamma}$ conjugating $\tilde{\varphi}$ with the rotation. We can assume that $\vert G(x_0) -\tilde{G}(x_0)\vert<\tau$ where $\tilde{G}$ is a lift of $\tilde{\gamma}$. Notice that $\Phi^n(x_0)=G^{-1}(G(x_0)+n\varrho)$ and $\tilde{\Phi}^n(x_0)=\tilde{G}^{-1}(\tilde{G}(x_0)+n\varrho)$ for $n=0,1,2...$. Thus the corresponding points on the orbits $\{\varphi^n(x)\}$ and $\{\tilde{\varphi}^n(x)\}$ remain $2\omega$-close (independently of $x_0$ and $n$).  Consequently we estimate:

\begin{multline}
\begin{aligned}
 \vert r_n(x_0)-\tilde{r}_n(x_0)\vert < & \frac{1}{2} M (\vert G^{-1}(G(x_0)+n\varrho)-G^{-1}(\tilde{G}(x_0)+n\varrho)\vert  +\\
   &\vert G^{-1}(\tilde{G}(x_0)+n\varrho)-\tilde{G}^{-1}(\tilde{G}(x_0)+n\varrho)\vert + \\
 & \vert G^{-1}(G(x_0)+(n-1)\varrho)-G^{-1}(\tilde{G}(x_0)+(n-1)\varrho)\vert +  \\
  &\vert G^{-1}(\tilde{G}(x_0)+(n-1)\varrho)-\tilde{G}^{-1}(\tilde{G}(x_0)+(n-1)\varrho)\vert)\\
  < &M (\tau+\omega)<2 M \omega, 
\end{aligned}\nonumber
\end{multline}


which ends the proof. \end{proof}

If we do not require that the rotation numbers of $\varphi$ and $\tilde{\varphi}$ are the same then the above follows for fixed $n$ (which is a simple observation). Namely,
\begin{remark}
Let $\varphi:S^1\to S^1$ be a minimal homeomorphism. Fix $\varepsilon>0$ and $n\in\mathbb{N}$. There exists a neighbourhood $U\subset C^{0}(S^1)$ of $\varphi$ such that for every other minimal homeomorphism $\tilde{\varphi}\in U$ we have
\begin{displaymath}
\sup_{x_0\in\mathbb{R}}\vert r_n(x_0)-\tilde{r}_n(x_0)\vert <\varepsilon,
\end{displaymath}
where $r_n(x_0)$ and $\tilde{r}_n(x_0)$ denote local radii of curvature of curves generated by $\varphi$ and $\tilde{\varphi}$, respectively. 
\end{remark}

\section{Discussion}\label{secDisc}
We have established a number of properties of curlicues generated by orientation preserving circle homeomorphisms. The first natural conclusion that we have drawn is that the geometrical properties of curlicues depend on the rationality of the rotation number of the generating circle homeomorphisms. Nevertheless, even for rational rotation number basic properties such as being bounded or not, might rather depend on the homeomorphism $h$ conjugating $\varphi$ with the corresponding rotation (if $\varphi$ is conjugated to the rotation), as follows from Examples \ref{Example 1} and \ref{Example 2}. On the other hand, for the irrational rotation number the relationship between the shape of the generated curve and the continuous solution of the corresponding cohomological equation seems to be an interesting observation. However, there are rather not explicit and easy to verify criteria assuring that such a solution exists (see e.g. \cite{ghys} for the special case, when the homeomorphism is the irrational rotation $\mathcal{R}_{\varrho}$ and the cohomological equation to be solved is $g(x)+u(x)=u(\mathcal{R}_{\varrho}(x))$, where $g:S^1\to \mathbb{R}$ is a given continuous function and a continuous function $u:S^1\to \mathbb{R}$ is the unknown of the problem).  Similarly, one cannot apriori determine whether the curve is superficial or not. Indeed, for bounded case or unbounded with non-zero Birkhoff average the situation is clear but in the remaining case it depends on more refined properties of the rotation number (Theorem \ref{TwierdzenieZbiorcze}). We know that the necessary condition for the curlicue to be bounded (and thus for the existence of a continuous solution of the cohomological equation) is the vanishing of the Birkhoff average. On the other hand, if the Birkhoff average does not vanish, then the curlicue is unbounded. 

Therefore, it would be interesting to characterize the case when the Birkhoff average equals zero but the induced curve is unbounded. Partially we answered this question in Theorem  \ref{TwierdzenieZbiorcze}, which allowed to establish superficiality and estimate the grow rate of such an unbounded curlicue. However, even providing a specific example of a minimal homeomorphism with vanishing Birkhoff average and unbounded curlicue (thus unbounded Birkhoff sums) seems a non-trivial task and further characterization of such curves may be a subject of further research. Similarly, this work might be a starting point for studying dynamically generated curlicues (and associated `walks', as mentioned in the Introduction), with, perhaps, some connections to the theory exponential sums and various Birkhoff averages. Generalization of these results for continuous circle mappings (instead of homeomorphisms) does not seem straightforward too.

\subsection*{Acknowledgements}
I would like to thank Ali Tahzibi from University of S\~{a}o Paolo at S\~{a}o Carlos for introducing me to the subject of curlicues and to Mario Ponce from Pontificia Universidad Cat\'{o}lica de Chile for fruitful discussions on circle homeomorphisms, curlicues and the relation between the cohomological equation and shape of curlicues during my visit at ICMC at S\~{a}o Carlos in January 2014.


\end{document}